\documentclass[reqno]{amsart}
\usepackage{amsmath,amscd,amsfonts,amsthm,amssymb,latexsym,mathrsfs,mathtools,caption}
\usepackage{tikz,hyperref,cleveref}
\usetikzlibrary{matrix,arrows,patterns,decorations.markings}
\usepackage{tikz,float,comment}
\usepackage{shuffle}
\numberwithin{equation}{section}
\usepackage[T1]{fontenc}
\usepackage{lmodern}
\usepackage{textcomp}
\usepackage{soul} 

\usepackage{hyperref}
\hypersetup{
    bookmarks=true,         
    unicode=false,          
    pdftoolbar=true,        
   pdfmenubar=true,        
   pdffitwindow=false,     
   pdfstartview={FitH},    
   pdftitle={},    
   pdfauthor={Author},     
    pdfcreator={Creator},   
    pdfproducer={Producer}, 
   pdfkeywords={keyword1} {key2} {key3}, 
   pdfnewwindow=true,      
    colorlinks=true,       
    linkcolor=blue,          
    citecolor=black,        
    filecolor=magenta,      
    urlcolor=black    
    }

\theoremstyle{plain}
   \newtheorem{theorem}{Theorem}[section]
   \newtheorem{proposition}[theorem]{Proposition}
   \newtheorem{lemma}[theorem]{Lemma}
   \newtheorem{corollary}[theorem]{Corollary}
   \newtheorem{problem}[theorem]{Problem}

\theoremstyle{definition}
   \newtheorem{definition}[theorem]{Definition}
   \newtheorem{example}[theorem]{Example}

\theoremstyle{remark}
   \newtheorem{remark}[theorem]{Remark}

\newcommand{\NN}{\mathbb{N}}

\newcommand{\one}{\hat{1}}
\newcommand{\zero}{\hat{0}}

\newcommand{\MM}{\mathrm{M}}

\newcommand{\HH}{\mathcal{H}}

\newcommand{\EE}{\mathcal{E}}

\newcommand{\FFF}{\mathbb{F}}

\newcommand{\TN}{\mathrm{TN}}

\newcommand{\ZZ}{\mathbb{Z}}

\newcommand{\RR}{\mathbb{R}}

\def\newop#1{\expandafter\def\csname #1\endcsname{\mathop{\rm
#1}\nolimits}}

\newop{diag}
\newop{AB}
\newop{DB}
\newop{B}
\newop{AT}
\newop{DT}
\newop{peak}
\newop{rank}
\newop{des}
\newop{Sym}
\newop{inv}
\newop{Orb}
\newop{Hom}
\newop{Re}

\author[P.~Br\"and\'en]{Petter Br\"and\'en}
\address{Department of Mathematics, KTH Royal Institute of Technology,
Sweden}

\author[L.~Saud]{Leonardo Saud Maia Leite}

\title[TN-matrices, chain enumeration and zeros of polynomials]{Totally nonnegative matrices, chain enumeration and zeros of polynomials}

\begin{document}
\definecolor{ududff}{rgb}{0.30196078431372547,0.30196078431372547,1.}
\definecolor{xdxdff}{rgb}{0.49019607843137253,0.49019607843137253,1.}
\definecolor{red}{rgb}{1,0,0}
\definecolor{ffqqff}{rgb}{1,0,1}
\definecolor{zzttqq}{rgb}{0.6,0.2,0}
\definecolor{ududff}{rgb}{0.30196078431372547,0.30196078431372547,1}
\definecolor{cqcqcq}{rgb}{0.7529411764705882,0.7529411764705882,0.7529411764705882}
\definecolor{black}{rgb}{0,0,0}

\begin{abstract}
We prove that any {lower unitriangular} and totally nonnegative matrix gives rise to a family of polynomials with only real zeros. This has consequences {for} problems in several areas of mathematics. We use it  to develop a general theory for chain enumeration in posets and zeros of chain polynomials.  The results obtained extend and unify results of the first author, Brenti, Welker and Athanasiadis. In the process we define a notion of $h$-vectors  for a large class of posets which generalize the notions of $h$-vectors associated to simplicial and cubical complexes. 
A consequence of our methods is a characterization of the convex hull of all characteristic polynomials of hyperplane arrangements of fixed dimension and over a fixed finite field. This may be seen as a refinement of the Critical Problem of Crapo and Rota. 

We also use the methods developed to answer an open problem posed by Forg\'acs and Tran on the real-rootedness of polynomials arising from certain bivariate rational functions. 
\end{abstract}


\maketitle
\thispagestyle{empty}
\tableofcontents
\section{Introduction}
Chain enumeration in partially ordered sets (posets) is a central topic in enumerative and algebraic combinatorics \cite[Chapter 3]{stanley2011enumerative}. The zeros of polynomials associated to chain enumeration {have been} studied frequently in the literature. {Special} attention has been given to the problem of determining if chain polynomials of posets are real-rooted for various classes of posets. For distributive lattices, this problem is equivalent to the Poset conjecture (Neggers-Stanley conjecture) for natural labelings, which was stated by Neggers \cite{neggers1978representations} in the seventies and disproved by Stembridge in \cite{stembridge2007counterexamples}.
Brenti and Welker \cite{brenti2008f} proved that the chain polynomials of the face lattices of simplicial polytopes are real-rooted, and  conjectured that the same is true for any polytope. Athanasiadis \cite{athanasiadis2021face} proved Brenti and Welker's conjecture for cubical polytopes. More recently, Athanasiadis, Kalampogia-Evangelinou and Douvropoulos \cite{athanasiadis2023two,athanasiadis2022chain} proved that the chain polynomials of subspace lattices, partition lattices of types $A$ and $B$, the lattices of flats of near-pencils and uniform matroids, rank-selected subposets of Cohen-Macaulay simplicial posets and noncrossing partition lattices associated to finite Coxeter groups are real-rooted.

We approach the problem of real-rooted chain polynomials through total positivity.  We prove a general theorem, Theorem~\ref{maint}, stating that any lower triangular totally nonnegative matrix whose diagonal entries are all equal to one gives rise to a family of real-rooted polynomials. For the special case when the matrix is a lower triangular Toeplitz matrix, Theorem~\ref{maint} implies that any P\'olya frequency sequence gives rise to an infinite family of real-rooted polynomials, see Theorem~\ref{PFT}. This is used to answer an open problem posed by Forg\'acs and Tran  \cite[Problem 13]{forgacs2016polynomials} on the real-rootedness of polynomials arising from certain bivariate rational functions, see Theorem~\ref{FTT}. We also apply Theorem~\ref{PFT} to chain enumeration in upper homogeneous (upho) posets, a class of posets recently introduced by Stanley \cite{stanley2020upho, stanley2024theorems}. We prove that {certain}  chain polynomials associated to Schur-positive upho poset \cite{fu2024monoid} are real-rooted, Theorem~\ref{uphot}.

In Section \ref{sec-tn} we identify a class of posets that we call $\mathrm{TN}$-posets, to which Theorem~\ref{maint} applies to prove that such posets have real-rooted chain polynomials. The class of  $\mathrm{TN}$-posets is closed under rank selection. Hence we get for free that the chain polynomials associated {to rank selected subposets} of a $\mathrm{TN}$-poset are real-rooted. Examples of $\mathrm{TN}$-posets are Boolean algebras, lattices of subspaces of a finite dimensional vector space, $r$-cubical lattices and dual partition lattices. Given a $\mathrm{TN}$-poset $P$ we introduce, in Section \ref{section-subdivision}, a notion of $h$-vectors relative to $P$. For Boolean complexes this notion reduces to the usual definition of $h$-{vector}, and for cubical complexes our definition reduces to Adin's cubical $h$-vector \cite{adin1996new}. We prove in Theorem~\ref{rpossel} that if the rank generating polynomial of a  $P$-poset $Q$ (for example an order ideal of $P$) has nonnegative $h$-vector, then the chain polynomial of $Q$ is real-rooted. This is a vast generalization of a theorem of Brenti and Welker \cite{brenti2008f} who proved that if a Boolean cell complex $\Delta$ has nonnegative $h$-vector, then the $f$-polynomial of the barycentric subdivision of $\Delta$ is real-rooted. Equivalently, the chain polynomial of $\Delta$ is real-rooted.

In Section \ref{sec-critical} we apply the theory developed in previous sections to  $q$-posets, a $q$-analog of simplicial complexes first considered by Rota \cite{Rota71}, and more extensively studied by Alder \cite{alder2010q}. These are posets for which every principal order ideal is isomorphic to the lattice of subspaces of a finite dimensional vector space over $\mathbb{F}_q$. We introduce an $h$-vector for $q$-posets, based on the theory developed in Section~\ref{section-subdivision}, and thus answer an open problem posed by Alder \cite{alder2010q}. We prove that the rank generating polynomial of any shellable $q$-poset has a nonnegative $h$-vector. In the process we provide a characterization of the convex hull of all characteristic polynomials of hyperplane arrangements of fixed dimension over a fixed finite field, Corollary \ref{ch-cor}.  This may be seen as a refinement of the Pritical Problem of Crapo and Rota \cite{crapo1970foundations}. In Theorem~\ref{qmrr} we prove that the chain polynomial of any $q$-matroid is real-rooted.

In Section \ref{sec-r-cubical} we consider $r$-cubical posets. We introduce $h$-vectors for $r$-cubical posets which generalize the h-vectors for cubical complexes defined by Adin \cite{adin1996new}. We prove that C-shellable $r$-cubical posets have nonnegative $h$-vectors, and that the chain polynomial of any $r$-cubical poset with a nonnegative $h$-vector is real-rooted. This generalizes Athanasiadis' theorem for the case of cubical complexes \cite{athanasiadis2021face}. 

In Section \ref{sec-part} we study the infinite dual partition lattice. We introduce notions of $h$-vectors and shellability for so called partition posets, and prove that chain polynomials of shellable partition posets are real-rooted. 

In a sequel \cite{BL2} to this paper we apply the theory developed here to chain enumeration in geometric lattices. For example we prove that the chain polynomials of Dowling lattices and the lattices of flats of paving matroids are real-rooted.

\section{Totally nonnegative and resolvable matrices}
\label{sec-tnm}
A {\emph{lower unitriangular matrix}} is a lower triangular matrix whose diagonal entries are all equal to one. In this section we will prove that for lower unitriangular matrix, total nonnegativity is equivalent to the resolvability of the row generating polynomials of $R$. This will enable us to  prove in Section \ref{sec-interlacing} that such matrices give rise to a family of real-rooted polynomials.

For $N \in \NN\cup\{\infty\}$, let $\Gamma_N$ be the directed graph on $\{ (i,j) \in \NN^2: j\leq i \leq N\}$ with edges 
$$
(i,j) \to (i,j+1) \ \ \ \mbox{ and } \ \ \ (i+1,j) \to (i,j), 
$$
and let $\lambda=(\lambda_{i,j})_{0\leq j \leq i < N}$ be an array of nonnegative numbers.   
\begin{figure}
\centering
\begin{tikzpicture}[line cap=round,line join=round,>=triangle 45,x=1cm,y=1cm,scale=1.4]
    \draw [color=cqcqcq,, xstep=0.5cm,ystep=0.5cm];
    \clip(-1.1480869706791401,-6.160766230552031) rectangle (7,0.7001462190018983);
    \draw [->,line width=0.3pt] (0,0) -- (5.5,0);
    \draw [->,line width=0.3pt] (0,0) -- (0,-5.5);
    \draw [line width=0.3pt] (0,-1) -- (1,-1);
    \draw [line width=0.3pt] (0,-2) -- (2,-2);
    \draw [line width=0.3pt] (0,-3) -- (3,-3);
    \draw [line width=0.3pt] (0,-4) -- (4,-4);
    \draw [line width=0.3pt] (0,-5) -- (5,-5);
    \draw [line width=0.3pt] (1,-5) -- (1,-1);
    \draw [line width=0.3pt] (2,-5) -- (2,-2);
    \draw [line width=0.3pt] (3,-5) -- (3,-3);
    \draw [line width=0.3pt] (4,-5) -- (4,-4);
    \draw [line width=0.3pt,domain=0:5.5] plot(\x,{(-0-2*\x)/2});
    \draw [->,-latex] (0,-1) -- (0,-0.4);
    \draw [->,-latex] (0,-1) -- (0.6,-1);
    \draw [->,-latex] (0,-2) -- (0,-1.4);
    \draw [->,-latex] (0,-2) -- (0.6,-2);
    \draw [->,-latex] (1,-2) -- (1,-1.4);
    \draw [->,-latex] (1,-2) -- (1.6,-2);
    \draw [->,-latex] (0,-3) -- (0,-2.4);
    \draw [->,-latex] (0,-3) -- (0.6,-3);
    \draw [->,-latex] (1,-3) -- (1,-2.4);
    \draw [->,-latex] (1,-3) -- (1.6,-3);
    \draw [->,-latex] (2,-3) -- (2,-2.4);
    \draw [->,-latex] (2,-3) -- (2.6,-3);
    \draw [->,-latex] (0,-4) -- (0,-3.4);
    \draw [->,-latex] (0,-5) -- (0,-4.4);
    \draw [->,-latex] (0,-4) -- (0.6,-4);
    \draw [->,-latex] (0,-5) -- (0.6,-5);
    \draw [->,-latex] (1,-4) -- (1,-3.5);
    \draw [->,-latex] (1,-4) -- (1.6,-4);
    \draw [->,-latex] (2,-4) -- (2,-3.4);
    \draw [->,-latex] (2,-4) -- (2.6,-4);
    \draw [->,-latex] (3,-4) -- (3,-3.4);
    \draw [->,-latex] (3,-4) -- (3.6,-4);
    \draw [->,-latex] (1,-5) -- (1,-4.4);
    \draw [->,-latex] (1,-5) -- (1.6,-5);
    \draw [->,-latex] (2,-5) -- (2,-4.4);
    \draw [->,-latex] (2,-5) -- (2.6,-5);
    \draw [->,-latex] (3,-5) -- (3,-4.4);
    \draw [->,-latex] (3,-5) -- (3.6,-5);
    \draw [->,-latex] (4,-5) -- (4,-4.4);
    \draw [->,-latex] (4,-5) -- (4.6,-5);
    \draw (-0.05404958007457205,0.42200111969565796) node[anchor=north west] {0};
    \draw (0.9658157840483304,-0.4773347013945192) node[anchor=north west] {1};
    \draw (1.9671381415508162,-1.4508425489663608) node[anchor=north west] {2};
    \draw (2.9777320023635103,-2.350178370056538) node[anchor=north west] {3};
    \draw (3.96051135324558,-3.4164012507304595) node[anchor=north west] {4};
    \draw (4.9525622074378575,-4.389909098302301) node[anchor=north west] {5};
    \draw (-0.4434527191033166,0.11604151045879355) node[anchor=north west] {0};
    \draw (-0.4805387323441494,-0.9038238536640878) node[anchor=north west] {1};
    \draw (-0.4712672290339412,-1.9236892177869693) node[anchor=north west] {2};
    \draw (-0.4527242224135248,-2.925011575289435) node[anchor=north west] {3};
    \draw (-0.4527242224135248,-3.898519422861276) node[anchor=north west] {4};
    \draw (-0.4712672290339412,-4.918384786984157) node[anchor=north west] {5};
    \draw (-0.5,-0.3) node[anchor=north west] {$\lambda_{00}$};
    \draw (-0.5,-1.3) node[anchor=north west] {$\lambda_{10}$};
    \draw (-0.5,-2.3) node[anchor=north west] {$\lambda_{20}$};
    \draw (-0.5,-3.3) node[anchor=north west] {$\lambda_{30}$};
    \draw (-0.5,-4.3) node[anchor=north west] {$\lambda_{40}$};
    \draw (0.5,-1.3) node[anchor=north west] {$\lambda_{11}$};
    \draw (0.5,-2.3) node[anchor=north west] {$\lambda_{21}$};
    \draw (0.5,-3.3) node[anchor=north west] {$\lambda_{31}$};
    \draw (0.5,-4.3) node[anchor=north west] {$\lambda_{41}$};
    \draw (1.5,-4.3) node[anchor=north west] {$\lambda_{42}$};
    \draw (2.5,-4.3) node[anchor=north west] {$\lambda_{43}$};
    \draw (3.5,-4.3) node[anchor=north west] {$\lambda_{44}$};
    \draw (1.5,-3.3) node[anchor=north west] {$\lambda_{32}$};
    \draw (2.5,-3.3) node[anchor=north west] {$\lambda_{33}$};
    \draw (1.5,-2.3) node[anchor=north west] {$\lambda_{22}$};
    \draw (5.6479249557034725,0.2365710534914977) node[anchor=north west] {\Large $j$};
    \draw (-0.07259258669498846,-5.446860475666013) node[anchor=north west] {\Large $i$};
    \begin{scriptsize}
        \draw [fill=black] (0,0) circle (1.0pt);
        \draw [fill=black] (2,-2) circle (1.0pt);
        \draw [fill=black] (0,-0.9888560254288775) circle (1.0pt);
        \draw [fill=black] (0,-2.0005086637613667) circle (1.0pt);
        \draw [fill=black] (1,-2) circle (1.0pt);
        \draw [fill=black] (0,-3) circle (1.0pt);
        \draw [fill=black] (1,-3) circle (1.0pt);
        \draw [fill=black] (2,-3) circle (1.0pt);
        \draw [fill=black] (0,-4) circle (1.0pt);
        \draw [fill=black] (0,-5) circle (1.0pt);
        \draw [fill=black] (1,-4) circle (1.0pt);
        \draw [fill=black] (2,-4) circle (1.0pt);
        \draw [fill=black] (3,-4) circle (1.0pt);
        \draw [fill=black] (1,-5) circle (1.0pt);
        \draw [fill=black] (2,-5) circle (1.0pt);
        \draw [fill=black] (3,-5) circle (1.0pt);
        \draw [fill=black] (4,-5) circle (1.0pt);
        \draw [fill=black] (1,-1) circle (1.0pt);
        \draw [fill=black] (3,-3) circle (1.0pt);
        \draw [fill=black] (4,-4) circle (1.0pt);
        \draw [fill=black] (5,-5) circle (1.0pt);
    \end{scriptsize}
\end{tikzpicture}
\caption{The directed graph $\Gamma_N$ with its weights.}
\label{fig:Gamma}
\end{figure}
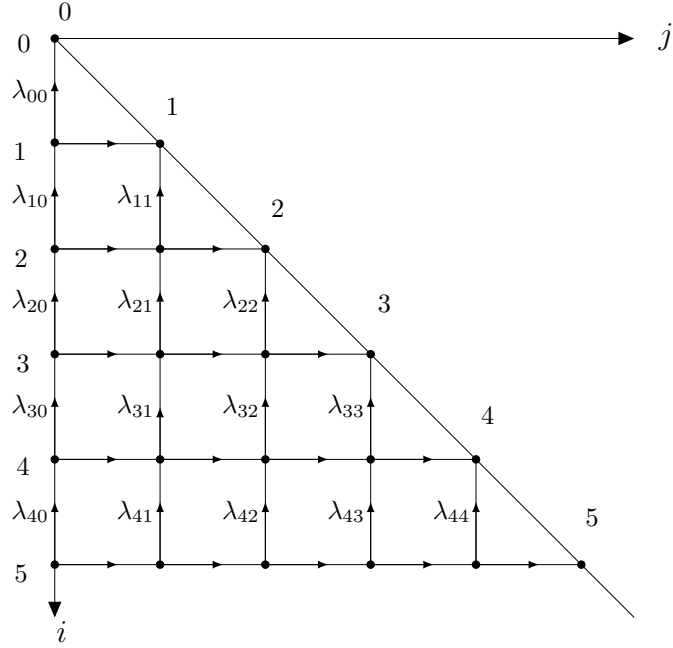

Attach the weight $\lambda_{i,j}$ to the vertical edge $(i+1,j) \to (i,j)$, and attach the weight $1$ to each horizontal edge, see Fig. \ref{fig:Gamma}. Let further $r_{n,k}(\Gamma_N, \lambda)$ be the weighted sum of all paths from $(n,0)$ to $(k,k)$, where the weight of a path is the product of the weights of the edges used in the path. Recall that a matrix with real entries is \emph{totally nonnegative} (TN) if all of its minors are nonnegative. By the Lindstr\"om-Gessel-Viennot Lemma~\cite[Thm. 7.16.1]{stanley2011enumerative}, the matrix $R(\Gamma_N, \lambda)= (r_{n,k}(\Gamma_N, \lambda))_{n,k=0}^N$ is totally nonnegative. By Whitney's reduction Theorem~\cite{whitney1952reduction}, the converse is also true. 

\begin{theorem}\label{WL}
Let $R=(r_{n,k})_{n,k=0}^N$, $N \in \NN\cup\{\infty\}$, be a lower {unitriangular matrix}. Then $R$ is totally nonnegative if and only if there is an array of nonnegative numbers $\lambda=(\lambda_{i,j})_{0\leq j \leq i < N}$ such that $R= R(\Gamma_N, \lambda)$. Moreover we may choose $\lambda$ so that 
\begin{equation}\label{limp}
\lambda_{n,k}=0 \ \ \ \mbox{ implies }  \ \ \ \lambda_{n+1,k}=0, 
\end{equation}
for all $0\leq k \leq n <N-1$, and then $\lambda$ is unique.
\end{theorem}
\begin{proof}
It suffices to prove the theorem for finite $N$. One direction is the Lindstr\"om-Gessel-Viennot lemma. 

We prove the existence of $\Gamma_N$, $N \in \NN$, and $\lambda$ by induction on $N$. Let $R=(r_{n,k})_{n,k=0}^N$ be a lower {unitriangular matrix}, and let $m+1$ be the {smallest} index for which $r_{m+1,0}=0$. Whitney's reduction theorem says that $R$ is $\mathrm{TN}$ if and only if $r_{j,0}=0$ for each $j>m$, and the matrix $\tilde{R}= (\tilde{r}_{n,k})_{n,k=0}^{N-1}$, where $\tilde{r}_{n,k}= r_{n+1,k+1}- \mu_n r_{n,k+1}$ and 
$$
\mu_n= \begin{cases}
r_{n+1,0}/r_{n,0} &\mbox{ if } n \leq m \\
0 &\mbox{ otherwise}
\end{cases}\,,
$$
is $\mathrm{TN}$. 

If $\lambda$ satisfies \eqref{limp}, then 
$$
r_{n+1,k+1}(\Gamma_N, \lambda)= \lambda_{n,0} r_{n,k+1}(\Gamma_N, \lambda)+ r_{n,k}(\Gamma_{N-1}, \tilde{\lambda}), 
$$
where $\tilde{\lambda}_{n,k}=\lambda_{n+1,k+1}$, and
$$
\lambda_{n,0} = \begin{cases}
    r_{n+1,0}(\Gamma_N, \lambda)/r_{n,0}(\Gamma_N, \lambda) &\mbox{ if } n \leq m \\
    0 &\mbox{ otherwise}
\end{cases}.
$$

Suppose $R$ is $\mathrm{TN}$. By induction there exist $\gamma$ satisfying \eqref{limp} such that $\tilde{r}_{n,k}= r_{n,k}(\Gamma_{N-1}, \gamma)$. But then 
$r_{n,k}= r_{n,k}(\Gamma_{N}, \lambda)$, where $\lambda$ is obtained from $\gamma$ by adding a first column equal to $\mu$ to the left of $\gamma$. 

The uniqueness of $\lambda$ follows by induction.
\end{proof}

\begin{definition}
\label{resolv}
Let $R=(r_{n,k})_{n,k=0}^N$, $N \in \NN \cup\{\infty\}$, be a lower {unitriangular matrix}, and let $R_n(t) = \sum_{k=0}^n r_{n,k} t^k$ be the generating polynomial of the $n$th row. 
The matrix $R$ (and the sequence $\{R_n(t)\}_{n=0}^N$) is called \emph{resolvable} if there is a matrix $(\lambda_{n,k})_{0\leq k\leq n <N}$ of nonnegative numbers, and an array of monic polynomials $(R_{n,k}(t))_{0\leq k \leq n \leq N}$ such that 
\begin{itemize}
\item $R_{n,0}(t)=R_n(t)$ and $R_{n,n}(t) =t^n$ for all $0\leq n \leq N$,
\item $t^k$ divides  $R_{n,k}(t)$ for all $0 \leq k \leq n \leq N$, and 
\item if $0\leq k \leq n <N$, then 
\begin{equation}\label{r-pasc}
R_{n+1,k}(t)=R_{n+1,k+1}(t)+ \lambda_{n,k} R_{n,k}(t). 
\end{equation}
\end{itemize}
\end{definition}
If the matrix $R$ is resolvable, then we say that the polynomials $(R_{n,k}(t))_{0\leq k \leq n \leq N}$ \emph{resolve} $R$. Notice that $(\lambda_{n,k})_{0\leq k\leq n <N}$ uniquely determines $(R_{n,k}(t))_{0\leq k \leq n \leq N}$.
\begin{example}
The identity matrix is resolvable with $R_{n,k}(t)=t^n$ and $\lambda_{n,k}=0$ for all $0 \leq k \leq n$. 
\end{example}

\begin{example}\label{forex}
    Let $(x_0, x_1, x_2, \ldots)$ be a sequence of real numbers such that $x_i \geq 1$ for each $i$, and consider the matrix $R = (r_{n,k})_{n,k=0}^\infty$  where
    \begin{equation}
    \label{maximal-matrix}
        r_{n,k} =
        \begin{cases}
         x_k &\mbox{ if } k<n, \\
            1 &\mbox{ if } k=n, \\
            0 &\mbox{ if } k > n.
        \end{cases}
    \end{equation}
    Then $R$ is resolvable with 
    $$
        R_{n,k} (t) =
        \begin{cases}
            R_n (t) &\mbox{ if } k=0, \\
            t^n + (x_{n-1} -1) t^{n-1} &\mbox{ if } 1 \leq k \leq n-1, \\
            t^n &\mbox{ if } k = n,
        \end{cases}
    $$
    and
    $$
        \lambda_{n,k} =
        \begin{cases}
        x_0 &\mbox{ if } n=k=0,\\
            1 &\mbox{ if } n>k=0, \\
            0 &\mbox{ if } 1 \leq k \leq n-1, \\
            x_n - 1 &\mbox{ if } k=n.
        \end{cases}
    $$
\end{example}

\begin{example}
The matrix $\left(\binom n k \right)_{n,k=0}^\infty$ is resolvable with $R_{n,k}(t)= t^k(1+t)^{n-k}$ and $\lambda_{n,k}=1$ for all $0 \leq k \leq n$.
\end{example}

Notice that, by \eqref{r-pasc},
\begin{equation}
\label{r-sum}
    R_{n+1,k}(t)= t^{n+1}+ \sum_{j \geq k} \lambda_{n,j}R_{n,j}(t), \ \ \ \ 0\leq k \leq n<N, 
\end{equation}
if $\{R_n(t)\}_{n=0}^N$ is resolvable.

The next theorem characterizes resolvability in terms of totally nonnegative matrices and ``quantum'' real-rooted polynomials. 
\begin{theorem}\label{eqcon}
    Let $R=(r_{n,k})_{n,k=0}^N$ be a lower {unitriangular matrix}. The following are equivalent:
    \begin{enumerate}
        \item $R$ is resolvable,
        \item There are linear diagonal operators $\alpha_i : \RR[t] \to \RR[t]$, $1 \leq i \leq N$,  such that 
        $$
            \alpha_i(t^k)= \alpha_{i,k}t^k, \mbox{ where } \alpha_{i,k} \geq 0 \mbox{ for all } i,k, 
        $$
        and 
        $$
            R_n(t) = (t+ \alpha_1)(t+\alpha_2) \cdots (t+\alpha_n) 1.
        $$
        \item $R$ is $\mathrm{TN}$. 
    \end{enumerate} 
    Moreover if (2) is satisfied, then $R_{n,k}(t) = (t+\alpha_1)\cdots (t+\alpha_{n-k})t^k$ and $\lambda_{n,k}= \alpha_{n+1-k,k}$. 
\end{theorem}

\begin{proof}

Suppose (2) holds, and let  $R_{n,k}(t) = (t+\alpha_1)\cdots (t+\alpha_{n-k})t^k$. Then 
\begin{align*}
R_{n+1,k}(t)-R_{n+1,k+1}(t) &= (t+\alpha_1)\cdots (t+\alpha_{n-k})( (t+\alpha_{n+1-k})t^k-t^{k+1})\\
&=\alpha_{n+1-k,k}R_{n,k}. 
\end{align*}
Hence $R$ satisfies (1) with $\lambda_{n,k}= \alpha_{n+1-k,k}$. Conversely if $R$ satisfies (1), then $R$ satisfies (2) with 
$\alpha_{i,k}=\lambda_{k+i-1,k}$.  

Now, suppose $R$ is a $\mathrm{TN}$-matrix. Then, by Theorem~\ref{WL}, $R=R(\Gamma_N, \lambda)$ for some nonnegative array $\lambda$. Define $(R_{n,k}(t))_{0\leq k \leq n \leq N}$ as follows. Let $$R_{n,k}(t)=\sum_{j=0}^n r_{n,j}^k(\lambda) t^j,$$ where $r_{n,j}^k(\lambda)$ is the $\lambda$-weighted sum over {all} paths from $(n,k)$ to $(j,j)$ in $\Gamma$. {Then} $R_{n,k}(t)$ {satisfies} Definition \ref{resolv} by construction. 
Conversely if  the polynomials $R_{n,k}(t)$ and $\lambda$ satisfy Definition \ref{resolv}, then $R=R(\Gamma_N, \lambda)$ since $\lambda$ uniquely determines $R_{n,k}(t)$.
\end{proof}

\begin{example}
Let $x_1,x_2, \ldots$ be nonnegative real numbers, and consider the matrix $R=(e_k(x_1,\ldots, x_n))_{n,k=0}^\infty$, where $e_k(x_1,\ldots,x_n)$ is the $k$th \emph{elementary symmetric polynomial} in the variables $x_1,\ldots, x_n$. Then $R_n(t) = (t+ \alpha_1)(t+\alpha_2) \cdots (t+\alpha_n) 1$, where $\alpha_i(t^k)= x_i t^k$ for all $i,k$. Hence $R$ is $\mathrm{TN}$ and $R_{n,k}(t)= t^k (t+x_1) \cdots (t+x_{n-k})$. 

\end{example}

\begin{remark}\label{thethe}
Henceforth we will always assume that the numbers $\lambda_{n,k}$ satisfy \eqref{limp}, so that the polynomials $R_{n,k}(t)$ associated to a resolvable matrix $R$ are uniquely determined. 
\end{remark}

\section{Interlacing sequences of polynomials from resolvable matrices}
\label{sec-interlacing}
In this section we will use resolvability to a prove a general theorem, Theorem~\ref{mainUTP}, which for any lower triangular $\mathrm{TN}$-matrix with all diagonal entries equal to one produces a family of real-rooted and mutually interlacing polynomials. This will be our main tool in forthcoming sections. 

Suppose $f(t), g(t) \in \RR[t]$ are real-rooted polynomials with positive leading coefficients, and that 
$$
\cdots \leq \alpha_3 \leq \alpha_2 \leq \alpha_1  \ \  \mbox{ and } \ \ \cdots \leq \beta_3 \leq \beta_2 \leq \beta_1 
$$
are the zeros of $f(t)$ and $g(t)$, respectively. We say that the zeros of $f(t)$ \emph{interlace} those of $g(t)$ if 
$$
 \cdots \leq \alpha_3 \leq \beta_3 \leq \alpha_2 \leq \beta_2 \leq \alpha_1 \leq \beta_1, 
$$
and we write $f(t) \prec g(t)$. In particular the degrees of $f(t)$ and $g(t)$ differ by at most one.  By convention we also write $0 \prec f(t)$ and $f(t) \prec 0$ for any real-rooted polynomial $f(t)$. A sequence $\{f_i(t)\}_{i=0}^n$ of real-rooted polynomials with nonnegative coefficients is said to be an \emph{interlacing sequence} if $f_i(t) \prec f_j(t)$ for all $i<j$. Let $\mathcal{P}_n$ be the set of all interlacing sequences $\{f_i(t)\}_{i=0}^n$ of polynomials.

\begin{lemma}\label{basint}
Let $\{f_i(t)\}_{i=0}^n \in \mathcal{P}_n$.  
\begin{enumerate}
\item If  $\{\lambda_i\}_{i=0}^n \in \RR_{\geq 0}^{n+1}$, then 
$$
f_0(t) \prec \lambda_0 f_0(t) +  \lambda_1 f_1(t)+\cdots +  \lambda_n f_n(t) \prec f_n(t).
$$
\item If $f_i(t)$ is  of degree $d$ and its zeros all lie in $[-1,0]$ for each $0\leq i \leq n$, then the sequence $\{g_i(t)\}_{i=0}^{n+1}$ defined by 
$$
g_k(t) = t\sum_{j=0}^{k-1}  f_{j}(t) + (1+t) \sum_{j = k}^n  f_{j}(t)
$$
is an interlacing sequence of polynomials whose zeros all lie in $[-1,0]$. 
\end{enumerate}
\end{lemma}
\begin{proof}
The first result is standard, see e.g. \cite[Lemma 2.6]{borcea2010multivariate}.  
For the second result,  let $h_j(t)= (1-t)^df_j(t/(1-t))$ and $r_j(t) = (1-t)^{d+1}g_j(t/(1-t))$. Then $\{h_j(t)\}_{j=0}^n$ is an interlacing sequence of polynomials with nonnegative coefficients, since the zeros of $f_j(t)$ all lie in $[-1,0]$. Moreover, 
$$
r_k(t) = t\sum_{j=0}^{k-1}h_j(t)+  \sum_{j=k}^nh_j(t),
$$
and hence $\{r_j(t)\}_{j=0}^{n+1}$ is interlacing by e.g. \cite[Corollary 7.8.7]{branden2015unimodality}. Since $g_j(t)= (1+t)^{d+1}r_j(t/(1+t))$, the result follows.
\end{proof}

For $N \in  \NN\cup \{\infty\}$, let $\RR_N[t]$ be the linear space of all polynomials in $\RR[t]$ of degree at most $N$. 
\begin{definition}\label{subop}
Let $R=(r_{n,k})_{n,k=0}^N$, where $N \in \NN\cup \{\infty\}$, be a lower {unitriangular matrix}.

The \emph{chain polynomials} associated to $R$ are the polynomials in $\{p_n(t)\}_{n=0}^N$ defined by $p_0(t)=1$, and 
\begin{equation}\label{polreceq}
p_n(t)= t \sum_{k=0}^{n-1}r_{n,k}p_k(t), \ \ \ 0<n \leq N. 
\end{equation}  

The \emph{subdivision operator} associated to $R$ is the linear map  $\EE : \RR_N[t] \to \RR_N[t]$ defined by $\EE(t^n)= p_n(t)$ for each $0 \leq n \leq N$. 
\end{definition}

Notice that the subdivision operator may alternatively be defined by $\EE(1)= 1$, and 
\begin{equation}
\label{fundE}
    \EE(t^n) = t\EE(R_n(t) -t^{n}), \ \ \ \mbox{ if } 0<n \leq N. 
\end{equation}

\begin{example}\label{binomet}
The subdivision operator for $\left(\binom n k \right)_{n,k=0}^\infty$ was considered in \cite{branden2015unimodality,brenti2008f}. This operator has the property that it maps  
the $f$-polynomial of a simplicial complex $\Delta$ to the $f$-polynomials of the barycentric subdivision of $\Delta$, see \cite[Ch. 7.3.3]{branden2015unimodality}. The chain polynomials in this case are given by 
\begin{equation}\label{snkk}
p_n(t) = \sum_{k=1}^n k! S(n,k) t^k, 
\end{equation}
where $S(n,k)$ is a Stirling number of the second kind \cite{stanley2011enumerative}. {This also follows from Proposition \ref{chain-interp-prop}, since the number of chains $\varnothing =S_0 \subset S_1 \subset S_2 \subset \cdots\subset S_k =[n]$ in the boolean lattice is $k!S(n,k)$.}

\end{example}

\begin{example}\label{forex2}
Let $x_0,x_1, \ldots$ be real numbers such that $x_i \geq 1$ for all $i$, and consider the matrix $R$ in Example~\ref{forex}. 
 Then 
 $$
 p_n(t) =  x_0 t (1+x_1t) (1+x_2t)\cdots (1+x_{n-1}t), \ \ \ n \geq 1.
 $$
\end{example}

\begin{proposition}\label{chainform}
Let $\mathsf{p}=(p_n(t))_{n=0}^N$ be the vector of chain polynomials associated to $R$. Then 
$$
\mathsf{p}= R(t)^{-1}(1,0,0,\ldots)^T, 
$$
where $R(t)= I-t(R-I)$, and $I$ is the identity matrix. Moreover 
$$
p_n(t) = (-1)^{n}\det \big(R(t) [\{1,2,\ldots, n\}, \{0,1,\ldots, n-1\}]\big), 
$$
i.e., $(-1)^{n}$ times the minor of $R(t)$ whose rows are indexed by $\{1,2,\ldots, n\}$ and columns by $\{0,1,\ldots, n-1\}$. 
\end{proposition}
\begin{proof}
The recursion \eqref{polreceq} translates as $(1+t)\mathsf{p}-tR\mathsf{p}=(1,0,0,\ldots)^T$, which proves the first statement. The second statement follows from the first by Cramer's rule. 

\end{proof}

The next theorem is the most general theorem on real-rootedness in this paper. It produces a family of real-rooted polynomials {from} any totally nonnegative {unitriangular matrix}.  The case when $R=\left(\binom n k \right)_{n,k=0}^\infty$ was proved by the first author in \cite{branden2006linear}. 

\begin{theorem}\label{mainUTP}
Let $R=(r_{n,k})_{n,k=0}^N$ be a {unitriangular $\mathrm{TN}$-matrix}, and let $0 \leq n \leq N$. Let further $\{R_{n,k}(t)\}_{k=0}^n$ be  the family of polynomials afforded by Definition \ref{resolv} and Theorem~\ref{eqcon}. 

Then
$\{\EE(R_{n,k}(t))\}_{k=0}^n$ is an interlacing sequence of polynomials whose zeros all lie in the interval $[-1,0]$. 
\end{theorem}
\begin{proof}
The proof is by induction over $n$, the case $n=0$ being trivial. 
Assume 
$\{\EE(R_{n,k}(t))\}_{k =0}^n$ is an interlacing sequence of polynomials whose zeros lie in $[-1,0]$.  Then by \eqref{r-sum} and \eqref{fundE}, 
\begin{align*}
 \EE(R_{n+1,k}(t)) &= \EE(t^{n+1})+ \sum_{j \geq k} \lambda_{n,j}\EE(R_{n,j}(t)) \\
&= t\sum_{j \geq 0} \lambda_{n,j}\EE(R_{n,j}(t))+  \sum_{j \geq k} \lambda_{n,j} \EE(R_{n,j}(t)) \\
&= t \sum_{j =0 }^{k-1} \lambda_{n,j} \EE(R_{n,j}(t)) + (1+t) \sum_{j \geq k} \lambda_{n,j} \EE(R_{n,j}(t)). 
\end{align*}
The proof now follows from Lemma~\ref{basint} (2), since the sequence $\{\lambda_{n,j}\EE(R_{n,j}(t))\}_{j=0}^n$ is interlacing.
\end{proof}

\begin{theorem}
\label{maint}
    Let $R=(r_{n,k})_{n,k=0}^N$, where $N \in \NN\cup \{\infty\}$, be a lower {unitriangular matrix}. If $R$ is totally nonnegative, then for each $n \in \NN$, the zeros of the chain polynomial $p_n(t)$ defined by \eqref{polreceq} are real and located in the interval $[-1,0]$. Moreover, the zeros of $p_n(t)$ interlace those of $p_{n+1}(t)$ for each $0 \leq n <N$. 
\end{theorem}

\begin{proof}
By Theorem~\ref{mainUTP} it remains to prove that $p_{n}(t) \prec p_{n+1}(t)$ for $0 \leq n <N$. By the recursion in the proof of Theorem~\ref{mainUTP}, 
$$
p_{n+1}(t)= \EE(R_{n+1,n+1}(t))=  t \sum_{j =0 }^{n} \lambda_{n,j} \EE(R_{n,j}(t)).  
$$
Since the sequence $\{\EE(R_{n,k}(t))\}_{k \geq 0}$ is interlacing, Lemma~\ref{basint} (2) implies $p_{n+1}(t) \prec t p_{n}(t)$, which implies  $p_{n}(t) \prec p_{n+1}(t)$. 
\end{proof}

\section{P\'olya frequency sequences}
\label{sec-polya}

Recall that a sequence $\{a_i\}_{i=0}^\infty$ of real numbers is a \emph{P\'olya frequency sequence} if the Toeplitz matrix $(a_{i-j})_{i,j=0}^\infty$ is $\mathrm{TN}$, where $a_k=0$ if $k<0$. P\'olya frequency sequences were characterized by Aissen, Schoenberg,  Whitney and Edrei \cite{aissen1951generating} as follows: 

\begin{theorem}\label{ASWE}
    A sequence $\{a_i\}_{i=0}^\infty$ of real numbers is a P\'olya frequency sequence if and only if its generating function is of the form 
    \begin{equation}
    \label{PFS}
        \sum_{n=0}^\infty a_n x^n = C x^N e^{\gamma x} \prod_{i=1}^\infty \frac {1+ \alpha_ix} {1- \beta_ix},
    \end{equation}
    where $C, \gamma, \alpha_i, \beta_i$ are nonnegative real numbers, $N \in \NN$, and $\sum_{i=1}^\infty (\alpha_i+\beta_i)<\infty$. 
\end{theorem} 
\begin{remark}\label{polpf}
Let $P : \NN \rightarrow \RR$. Then $P$ is a polynomial of degree $d$ if and only if 
$$
\sum_{n=0}^\infty P(n) x^n = \frac {h(x)} {(1-x)^{d+1}},
$$
where $h(x)$ is a polynomial of degree at most $d$ for which $h(1) \neq 0$, see e.g. \cite[Corollary 4.3.1]{stanley2011enumerative}. Hence, by Theorem~\ref{ASWE}, the sequence $\{P(n)\}_{n =0}^\infty$, where $P$ is a polynomial of degree $d$, is a P\'olya frequency sequence if and only if all zeros of $h(x)$ are real and non-positive. 
\end{remark}

If $R$ is a lower triangular matrix and $i,j \in \NN$, consider the map $Z_{ij}^R : \NN \to \RR$ defined by  
$$
n \longmapsto (R^n)_{ij}, 
$$
that is, $Z_{ij}^R(n)$ is the $(i,j)$-entry of $R^n$. 

\begin{theorem}\label{Z-pos}
Suppose $R$ is a lower {unitriangular matrix}, and let $i\geq  j$ be nonnegative integers. 
\begin{itemize}
\item[(a)] 
$Z_{ij}^R(n)$ is a polynomial in $n$ of degree $i-j$. 
\item[(b)] If $R$ is $\mathrm{TN}$, then $\{Z_{ij}^R(n)\}_{n=0}^\infty$ is a P\'olya frequency sequence. 
\end{itemize}
\end{theorem}

\begin{proof}
The case when $j>0$ follows from the case when $j=0$ by observing that $(R^n)_{ij}= (S^n)_{i-j,0}$, where $S$ is the matrix obtained from $R$ by deleting the $j$ initial rows and columns of $R$. 

Suppose $j=0$. Let 
$$
f_i(t)= \sum_{n=0}^\infty Z_{i0}^R(n)t^n = \big( (I-tR)^{-1}\big)_{i0}.
$$
Since 
$$
I-tR= (1-t) \left(     \left(1+ \frac t {1-t} \right) I - \frac t {1-t} R     \right),
$$
it follows from Proposition~\ref{chainform} that 
$$
f_i(t)= \frac 1 {1-t} p_i\left(\frac t {1-t}\right), 
$$
where $p_i(t)$ is the polynomial defined by \eqref{polreceq}. 
Consider the polynomial $h_i(t)=(1-t)^i p_i(t/(1-t))$. Then  $h_i(t)$ has degree at most $i$, and 
$$
f_i(t)= \frac {h_i(t)}{(1-t)^{i+1}},  
$$
 so that (a) follows from Remark \ref{polpf}. 

If $R$ is $\mathrm{TN}$, then $p_i(t)$ is a polynomial of degree $i$ whose zeros all lie in $[-1,0]$ by Theorem~\ref{maint}. Hence all zeros of $h_i(t)$ are real and nonpositive, from which (b) follows by  Remark \ref{polpf}. 
\end{proof}

Let $g(x)=\sum_{n=1}^\infty b_n x^n \in \RR[[x]]$ be a formal power series with constant term equal to {zero}, and consider the expansion
$$
\frac 1 {1-tg(x)}= \sum_{n=0}^\infty g_n(t) x^n \in \RR[[x,t]].
$$
Then $g_n(t)$ is a polynomial of degree at most $n$ for each $n \in \NN$. We will now prove that for an important class of series $g(x)$, the polynomials $g_n(t)$ are all real-rooted.

\begin{theorem}
\label{PFT}
Let $f(x)$ be as in \eqref{PFS}, and consider the formal power series
\begin{equation}\label{pnp}
\frac 1 {1-t(f(x)-f(0))} = \sum_{n=0}^\infty r_n(t) x^n \in \RR[t][[x]]. 
\end{equation}
Then $r_n(t)$ is a real-rooted polynomial for each $n \in \NN$, and the zeros of $r_n(t)$ interlace those of $r_{n+1}(t)$. Moreover, if $f(0) \neq 0$, then all zeros of 
$r_n(t)$ lie in the interval $[-1/f(0),0]$.   
\end{theorem}

\begin{proof}
Suppose first that $f(0) \neq 0$, and let $r_n(f;t)$ be the polynomial $r_n(t)$ defined by \eqref{pnp}. Since 
$$
\frac 1 {1-t(f(x)-f(0))} = \frac 1  {1-tf(0) \left(\frac {f(x)}{f(0)}-1\right)},
$$
$r_n(f;t)= r_n(f/f(0); tf(0))$. Hence we may assume $f(0) =1$.

The matrix $R=(a_{i-j})_{i,j=0}^\infty$ is $\mathrm{TN}$, and hence the polynomials 
$p_n(t)$ defined by \eqref{polreceq} are $[-1,0]$-rooted polynomials by Theorem~\ref{maint}. Moreover $p_n(t) \prec p_{n+1}(t)$. Clearly, $p_0(t)=1$ and 
$
p_n(t)= -tp_n(t)+t\sum_{k=0}^{n}a_{n-k} p_k(t)$ for  $n \geq 1$, 
from which we deduce
$$
\frac 1 {1-t(f(x)-f(0)))} = \sum_{n=0}^\infty p_n(t) x^n. 
$$
Hence $r_n(t) = p_n(t)$, which proves the case when $f(0) \neq 0$ by  Theorem~\ref{maint}.

If $f(x)$ is as in \eqref{PFS}, with $f(0)=0$, then consider 
$$
f_{\epsilon}(x)=  \frac {(x+\epsilon)^N}{x^N} f(x), 
$$
where $\epsilon \geq 0$. 
Then, for fixed $n$, the coefficients of the degree $n$ polynomial $r_n(f_\epsilon;t)$ depend continuously on $\epsilon$. The conclusion follows from Hurwitz's theorem on the continuity of zeros. 
\end{proof}

\subsection{A problem of Forg\'acs and Tran}
The next theorem solves an open problem stated by Forg\'acs and Tran \cite[Problem 13]{forgacs2016polynomials}. 

\begin{theorem}\label{FTT}
Let $Q(x)$, where $Q(0)>0$, be a polynomial whose zeros are all real and positive, and let $r$ be a positive integer\footnote{Notice that Problem 13 is stated for all $r\geq 0$ in \cite{forgacs2016polynomials}. However for $r=0$, $q_n(t)$ fails to be a polynomial.}. Consider the power series 
$$
\frac 1 {Q(x)-tx^r} = \sum_{n=0}^\infty q_n(t) x^n. 
$$
Then $q_n(t)$ is a polynomial whose zeros are real and negative for each $n \in \NN$. 

Moreover, the zeros of $q_n(t)$ interlace those of $q_{n+1}(t)$ for each $n \in \NN$. 
\end{theorem}

\begin{proof}
The series 
$
f(x) = {x^r}/{Q(x)}
$
is of the form \eqref{PFS}. Hence the polynomials $r_n(t)$ in 
$$
\frac 1 {1-tf(x)} = \sum_{n=0}^\infty r_n(t) x^n
$$
are all real-rooted, and the zeros of consecutive polynomials interlace by Theorem~\ref{PFT}. Since 
$$
\sum_{n=1}^\infty r_n(t) x^n =  \frac {tx^r} {Q(x)-tx^r},
$$
we see that $q_n(t) = t^{-1}r_{n+r}(t)$, from which the theorem follows. 
\end{proof}

\section{\texorpdfstring{$\mathrm{TN}$}--posets}
\label{sec-tn}
For undefined poset terminology we refer to \cite{stanley2011enumerative}. 
In this section we study chain polynomials associated to what we call quasi-rank uniform posets.  Suppose $P$ is a locally finite poset. Define a \emph{quasi-rank function} $\rho : P \to \NN$ by 
$$
    \rho(x) = \max\{ k : x_0 <x_1 < \cdots <x_k=x, {\mbox{ where } x_i \in P \mbox{ for all } 0\leq i \leq k}\}, 
$$
{for each $x \in P$.}
The \emph{quasi-rank} of $P$ is $\rho(P)= \sup\{ \rho(x) : x \in P\} \in \NN\cup\{\infty\}$.  
\begin{definition}
\label{rank uniform}
Let $P$ be a locally finite poset.  
$P$ is called \emph{quasi-rank uniform} if
for any $x, y\in P$ with $\rho(x)=\rho(y)$,
$$
|\{ z \in \langle x \rangle : \rho(z) = k \}| = |\{ z \in  \langle y \rangle  : \rho(z) = k \}|, \ \ \mbox{ for all } 0 \leq k \leq \rho(x),  
$$
where $ \langle x \rangle = \{z \in P : z \leq x\}$. 
If  in addition  $\rho(y) = \rho(x)+1$ whenever $y$ covers $x$ in $P$, i.e., when $\rho$ is the rank function of $P$, then we say that $P$ is \emph{rank uniform}.   
\end{definition}

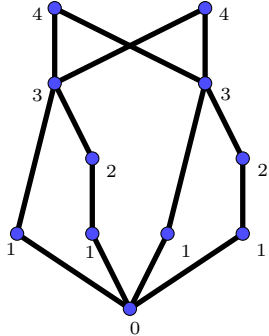
\begin{figure}[H]
\centering
\begin{tikzpicture}[line cap=round,line join=round,>=triangle 45,x=1cm,y=1cm,scale=0.5]
\clip(-5,-6) rectangle (10.339115239601984,3.395947613421145);
    \draw [line width=1pt] (-1,-3)-- (2,-5);
    \draw [line width=1pt] (1,-3)-- (2,-5);
    \draw [line width=1pt] (2,-5)-- (3,-3);
    \draw [line width=1pt] (2,-5)-- (5,-3);
    \draw [line width=1pt] (1,-3)-- (1,-1);
    \draw [line width=1pt] (1,-1)-- (0,1);
    \draw [line width=1pt] (-1,-3)-- (0,1);
    \draw [line width=1pt] (5,-3)-- (4,1);
    \draw [line width=1pt] (3,-3)-- (3,-1);
    \draw [line width=1pt] (3,-1)-- (4,1);
    \draw [line width=1pt] (0,3)-- (0,1);
    \draw [line width=1pt] (0,3)-- (4,1);
    \draw [line width=1pt] (4,1)-- (4,3);
    \draw [line width=1pt] (4,3)-- (0,1);
    \begin{scriptsize}
        \draw [fill=ududff] (2,-5) circle (5.0pt);
        \draw (1.8,-5.142979430587969) node[anchor=north west] {$0$};
        \draw [fill=ududff] (-1,-3) circle (5.0pt);
        \draw (-1.5,-3.046320446741683) node[anchor=north west] {$1$};
        \draw [fill=ududff] (1,-3) circle (5.0pt);
        \draw (0.6,-3.1384034426538507) node[anchor=north west] {$1$};
        \draw [fill=ududff] (1,-1) circle (5.0pt);
        \draw (0.1662321852744877,-0.9638280776511151) node[anchor=north west] {$2$};
        \draw [fill=ududff] (0,1) circle (5.0pt);
        \draw (-0.8,1.0407479102830028) node[anchor=north west] {$3$};
        \draw [fill=ududff] (3,-3) circle (5.0pt);
        \draw (3.1495582510750277,-3.0959035983866965) node[anchor=north west] {$1$};
        \draw [fill=ududff] (5,-3) circle (5.0pt);
        \draw (5.154134239009145,-3.06757036887526) node[anchor=north west] {$1$};
        \draw [fill=ududff] (3,-1) circle (5.0pt);
        \draw (3.2,-0.956744770273256) node[anchor=north west] {$2$};
        \draw [fill=ududff] (4,1) circle (5.0pt);
        \draw (4.204971050376029,1.0832477545501573) node[anchor=north west] {$3$};
        \draw [fill=ududff] (0,3) circle (5.0pt);
        \draw (-0.8,3.1940733531521617) node[anchor=north west] {$4$};
        \draw [fill=ududff] (4,3) circle (5.0pt);
        \draw (4.162471206108875,3.20823996790788) node[anchor=north west] {$4$};
        \end{scriptsize}
\end{tikzpicture}
\caption{A quasi-rank uniform poset, and the corresponding quasi-rank function.}
\label{fig:quasi-rank-uniform}
\end{figure}

We will primarily be interested in rank uniform posets, but find it more convenient working in the more general setting of quasi-rank uniform posets.

Let $P$ be quasi-rank uniform. 
If $x \in P$, $\rho(x)=n$ and $0 \leq k \leq n$, let 
$$
    r_{n,k}= r_{n,k}(P)=|\{ z \in  \langle x \rangle  : \rho(z) = k \}|, 
$$
and let $R= R(P)= (r_{n,k})_{n,k=0}^N$. 

\begin{proposition}\label{chain-interp-prop}
Let $P$ be a quasi-rank uniform poset, and let $n$ be a positive integer. The chain polynomial  $p_n(t)$, see \eqref{polreceq}, associated to the matrix $R(P)$ may be expressed as 
\begin{equation}\label{pnchain}
p_n(t) = \sum_{j \geq 1} |\{x_0 < x_1<\cdots < x_j=x : \rho(x_0)=0\}|t^j,  
\end{equation}
where $x$ is fixed element in $P$ for which $\rho(x)=n$.
\end{proposition}
\begin{proof}
 Let $P$ be quasi-rank uniform. Define polynomials $p_x(t)$, $x \in P$, by  $p_{x}(t)=1$ if $x$ is a minimal element, and 
 $$
 p_x(t)= \sum_{j \geq 1} |\{x_0 < x_1<\cdots < x_j=x : \rho(x_0)=0\}|t^j, \ \ \ \mbox{ if } \rho(x) > 0.
 $$
 If $\rho(x)> 0$, then 
 $$
 p_x(t)=t \sum_{y<x} p_y(t), 
 $$
 from which the proof follows by induction on $\rho(x)$. 
\end{proof}

We call the polynomial \eqref{pnchain} the  $n$th \emph{chain polynomial}\footnote{If $P$ has a least element $\zero$,  then the chain polynomial of the interval $[\zero, x]$ in $P$ is commonly defined as $t^{-1}(1+t)^{2}p_n(t)$.} of $P$. 

\begin{remark}\label{mob}
Let $P$ be a quasi-rank uniform poset with a least element $\zero$, and suppose $x \in P$ with $\rho(x)=n$. By \eqref{pnchain} and Philip Hall's Theorem~\cite[Prop. 3.8.5]{stanley2011enumerative}, $p_n(-1)= \mu_P(\zero,x)$, where $\mu_P$ is the M\"obius function of $P$. 

\end{remark}

\begin{definition}
\label{TN-poset}
    A quasi-rank uniform poset $P$ is a \emph{$\mathrm{TN}$-poset} if  the matrix $R(P)$ is totally nonnegative. 
\end{definition}

From Theorem~\ref{maint} and \eqref{pnchain} we deduce

\begin{theorem}
\label{TN-poset-chain}
    If $P$ is a $\mathrm{TN}$-poset, then for each $n\in \NN$ the zeros of the chain polynomial $p_n(t)$ are real and located in the interval $[-1,0]$. Moreover $p_{n}(t) \prec p_{n+1}(t)$. 
\end{theorem}

\begin{example}
Let $x_0,x_1,\ldots$ be positive integers with $x_0=1$, and recall the matrix $R$ considered in Examples \ref{forex} and \ref{forex2}. Let $P$ be the rank uniform poset that has exactly $x_i$ elements of rank $i$ for all $i$, and is such that $x<y$ whenever the rank of $x$ is smaller than the rank of $y$. Then $R(P)=R$, and $P$ is $\mathrm{TN}$ by Example \ref{forex}. Moreover 
 $$
 p_n(t) = t(1+x_1t) (1+x_2t)\cdots (1+x_{n-1}t), \ \ \ n \geq 1,
 $$
   by Example \ref{forex2}.  
\end{example}

\begin{remark}
Suppose $P$ is a quasi-rank uniform poset with a least element $\zero$, and let $x \in P$, $\rho(x)=k$. Then 
$$
Z_{k0}^R(n)= \left(R^n\right)_{k0}= |\{ \zero = x_0 \leq x_1 \leq \cdots \leq x_n =x\}|. 
$$
Hence $Z_{k0}^R(n)$ is equal to the \emph{zeta polynomial} $Z_{[\zero, x]}(n)$ of the interval $[\zero, x]$, see \cite[Section 3.12]{stanley2011enumerative}. Hence if $P$ is $\mathrm{TN}$, then $\{Z_{[\zero, x]}(n)\}_{n=0}^\infty$ is a P\'olya frequency sequence by Theorem~\ref{Z-pos}.
\end{remark}

If $S=\{s_0<s_1<\cdots\}$ is a subset of $\NN$, and $P$ is quasi-rank uniform let  
$$
P_S= \{ x \in P: \rho(x) \in S\}
$$
be the \emph{(quasi-)rank selected} subposet induced by $S$. Define $\rho_S : P_S \to \NN$ by $\rho_S(x)= k$ if $\rho(x)=s_k$. 

\begin{lemma}\label{rightr}
Let $P$ be a locally finite poset with quasi-rank function $\rho$, and let $S=\{s_0<s_1<\cdots\}$ be a subset of $\NN$. Then the quasi-rank function associated to $P_S$ is equal to $\rho_S$. 
\end{lemma}

\begin{proof}
Let $\rho'$ be the quasi-rank function of $P_S$, and let $x \in P_S$ be such that $\rho(x)=s_k$. Let $x_0 < x_1<\cdots < x_{s_k} =x$ be a chain of maximal length in $\langle x \rangle$. Then $\rho(x_j)=j$ for all $0\leq j \leq s_k$, since otherwise $\rho(x_j)>j$ for some $j$ and then we could find a longer chain in $\langle x \rangle$. Hence $x_{s_0} < x_{s_1} <\cdots < x_{s_k} =x$ is a chain in $P_S$, which proves $\rho'(x) \geq k = \rho_S(x)$. 

On the other hand if $\rho'(x)= j$ and $\rho_S(x)=k$, then there exists a chain $$x_{s_{i_0}} < x_{s_{i_1}} <\cdots < x_{s_{i_j}} =x,$$ 
in $P$ where $\rho(x_{s_{i_\ell}}) = s_{i_\ell} \in S$ for each $1 \leq \ell \leq j$. Moreover $i_j= k$, and hence $\{i_1, i_2,\ldots, i_j \} \subseteq \{1,2,\ldots, k\}$, which implies $j\leq k$.
\end{proof}

Since submatrices of $\mathrm{TN}$-matrices are $\mathrm{TN}$,  Lemma~\ref{rightr} implies the following result. 
\begin{proposition}
\label{propranksel}
    Let $P$ be a $\mathrm{TN}$-poset, and let $S$ be a set of nonnegative integers. Then $P_S$ is a $\mathrm{TN}$-poset. 
\end{proposition}

Let $P$ be a quasi-rank uniform poset with a least element $\zero$. Following \cite{stanley1976binomial}, we define $\mu_S(n)=\mu_{P_S}(\zero, x)$ where $\rho_S(x)=n$, and $\mu_{P_S}$ is the M\"obius function of $P_S$. If $S=\NN$, then we write $\mu_S(n)=\mu(n)$.  The next theorem generalizes \cite[Theorem~2.1]{stanley1976binomial}.

\begin{theorem}\label{mobius}
Let $P$ be a quasi-rank uniform poset with a least element $\zero$, and let $S=\{s_0<s_1< \cdots \}$ a set of nonnegative integers containing $0$. Then 
$$
\mu_S(n) = (-1)^n \det(R[\{s_1,\ldots, s_n\}, \{s_0, \ldots, s_{n-1}\}]). 
$$
\end{theorem}

\begin{proof}
It suffices to prove the theorem for $S=\NN$, since rank selection corresponds to taking principal submatrices of $R$. 

By Remark \ref{mob}, $p_n(-1)= \mu(n)$. The theorem now follows from Proposition~\ref{chainform}.  
\end{proof}

Theorem~\ref{mobius} provides a necessary condition for a poset to be $\mathrm{TN}$, namely that the M\"obius function alternates in sign. 

\begin{corollary}\label{altmu}
If $P$ is a $\mathrm{TN}$-poset with a least element $\zero$, then $(-1)^n \mu_S(n) \geq 0$ for all $S$ and $n$. 
\end{corollary}
Let us formulate Theorem~\ref{mobius} and Corollary \ref{altmu} in terms of flag $h$-vectors of quasi-rank uniform $\mathrm{TN}$-posets. Recall that a poset is bounded if it contains a $\zero$ and a $\one$. 
Let $P$ be a finite and bounded poset of rank $n$. The \emph{flag $f$-vector} of $P$ is the function $\alpha_P : 2^{[n-1]} \to \ZZ$ defined by 
$\alpha_P(S)$ is equal to the number of maximal chains of $P_S$. The  \emph{flag $h$-vector} of $P$  is the function $\beta_P : 2^{[n-1]} \to \ZZ$ defined by 
$$
\beta_P(S) = \sum_{T \subseteq S} (-1)^{|S \setminus T|}\alpha_P(T). 
$$
It is known that the flag $h$-vectors of Cohen-Macaulay posets are nonnegative, see \cite{BjornerCM}. The same is true for bounded $\mathrm{TN}$-posets. 
\begin{theorem}
Let $P$ be a quasi-rank uniform and bounded poset of quasi-rank $n$, and let $S=\{s_1,\ldots, s_k\} \subseteq [n-1]$. Then 
$$
\beta_P(S)  = \det(R[\{s_1,\ldots, s_k, n\}, \{0,s_1, \ldots, s_{k}\}]), 
$$
where $R=R(P)$. 

Hence if $P$ is $\mathrm{TN}$, then $\beta_P(S) \geq 0$ for all $S \subseteq [n-1]$. 
\end{theorem}
\begin{proof}
By Philip Hall's theorem,  $\beta_P(S) = (-1)^{n}\mu_{P_{S\cup\{0,n\}}}(\zero, \one)$. The theorem now follows from Theorem~\ref{mobius}.  
\end{proof}

\subsection{$\TN$-posets from $\TN$-matrices} 
Suppose $P$ is a quasi-rank uniform poset with $R(P)= (r_{n,k})_{k,n=0}^N$. Then all diagonal entries of $R(P)$ are equal to one, and $r_{n,k} \leq r_{n+1,k}$ for all $n<N$. The latter is true since if $\rho(x)=n+1$, then there exists $x' < x$ with $\rho(x')=n$. We shall now see that the converse is true.
\begin{theorem}\label{TNTTNP}
Let $N \in \NN \cup\{\infty\}$. 
Suppose $R=(r_{n,k})_{k,n=0}^N$  is a lower {unitriangular matrix} with nonnegative integer entries. If $r_{n,k} \leq r_{n+1,k}$ for all $k\leq n <N$, then there exists a quasi-rank uniform poset $P$, with a largest element $\one$, for which $R=R(P)$. 
\end{theorem}

\begin{proof}
We prove the theorem for $N \in \NN$ first. We prove by induction on $N \in \NN$ that there exists a poset $P_N$ with a greatest element $\one$ for which $R(P_N)=R$. The case when $N=0$ is trivial. Suppose true for some  $N \geq 0$. Consider a matrix $R_{N+1}= (r_{n,k})_{k,n=0}^{N+1}$  satisfying the hypothesis in the statement of the theorem. By induction there is a quasi-rank uniform poset $P_N$ on $X$ with $R(P_N)= (r_{n,k})_{k,n=0}^{N}$. For each $1 \leq n \leq N$, let $x_n \in P_N$ be a specified element of quasi-rank $n$.  Construct a poset $P_{N+1}$ on the disjunct union $X \cup Y_0 \cup Y_1 \cup \cdots \cup Y_{N+1}$, where $|Y_k|=r_{N+1,k}-r_{N,k}$ as follows.
\begin{enumerate}
\item Order $X$ as in $P_N$.
\item Let $Y_{N+1}$ consist of a single element which is declared to be the new largest element $\one_{N+1}$ in $P_{N+1}$.
\item Let $Y_0$ be a set of minimal elements in $P_{N+1}$ which are only related to $\one_{N+1}$. 
\item For $1 \leq n \leq N$, let $Y_n$ be an anti-chain in $P_{N+1}$. If $y \in Y_n$, then $y < \one_{N+1}$ and $x < y$ for all $x < x_n$ are the relations involving $y$. 
\end{enumerate}
It follows that $P_{N+1}$ is quasi-rank uniform and $R(P_{N+1})= R_{N+1}$. 

The case when $N = \infty$ follows from the construction above. Clearly $P_N$ is an order ideal of $P_{N+1}$ for all finite $N$. Hence $P= \cup_{N=0}^\infty P_N$ has the desired properties. 
\end{proof}

\subsection{Binomial and Sheffer posets}
Binomial and Sheffer posets are rank uniform posets  that have been studied frequently in combinatorics \cite{doubilet1972foundations,ehrenborg1995sheffer,reiner1993upper}. \emph{Binomial posets} are rank uniform posets $P$ with a least element $\zero$ for which  
any two intervals of the same length contain the same number of maximal chains. 
The \emph{factorial function} of $P$ is the function $B : \NN \to \NN$, where $B(n)$ is the number of maximal chains in an interval of length $n$ in $P$. 
More generally a \emph{Sheffer poset} is  a rank uniform poset $P$ for which 
\begin{itemize}
\item if $\rho(x)=\rho(y)=n$, then $[\zero, x]$ and $[\zero, y]$ have the same number $D(n)$ of maximal chains,
\item if $[x,u]$ and $[y,v]$, $\zero \not \in \{x,y\}$, are two intervals of the same length $n$, then they have the same number $B(n)$ of maximal chains.
\end{itemize}
In particular, binomial posets are Sheffer posets, with $B=D$.

\begin{example}
    The infinite Boolean algebra of finite subsets of $\NN$ is binomial with factorial function $B(n)=n!$.
\end{example}

\begin{example}
    Let $q$ be a prime power and let $\mathbb{B}(q)$ be the lattice of all finite subspaces of a vector space of infinite dimension over $\FFF_q$. Then $\mathbb{B}(q)$ is a binomial poset with 
    $$B(n)= \mathbf{n} !=  (1 + q + \cdots + q^{n-1})\cdot (1 + q + \cdots + q^{n-2})\cdots (1+q)\cdot 1. $$
\end{example}

\begin{example}
\label{def-r-cube}
    The infinite $r$-cubical lattice, $r \in \mathbb{Z}_{>0}$, is a Sheffer poset see \cite{ehrenborg1996r}. It is defined as the {locally finite part of the} infinite Cartesian product (with a least element $\hat{0}$ adjoined)
    $$
        \mathbf{C}_r=  \left( \prod_{i \geq 1} M_r \right) \sqcup \{ \hat{0} \},  
        $$
        of copies of $M_r$, where $M_r$ is the poset formed from an antichain on $[r]=\{1,\ldots,r\}$ and a largest element denoted $z$, see Fig. \ref{fig:M_r}. {In other words, each $x \in  \mathbf{C}_r \setminus \{\zero\}$ may be written uniquely  as $x=(\xi_1,\xi_2, \ldots)$ where $\xi_i \in M_r$ for each $i$, and $\xi_i \in[r]$ for all but finitely many $i$.} The factorial functions for  $\mathbf{C}_r$ are  $B(n) = n!$ and $D(n)= r^{n-1}(n-1)!$, see \cite{ehrenborg1996r}.

\begin{figure}[H]
    \centering
    \begin{tikzpicture}[line cap=round,line join=round,>=triangle 45,x=1cm,y=1cm]
        \clip(3.0,-1) rectangle (6.6,3);
        \draw [line width=1pt] (5,2)-- (4.5,0);
        \draw [line width=1pt] (5,2)-- (4,0);
        \draw [line width=1pt] (5,2)-- (6,0);
        \begin{scriptsize}
            \draw [fill=xdxdff] (4.5,0) circle (2.5pt);
            \draw [fill=xdxdff] (4,0) circle (2.5pt);
            \draw (4.9,0.2) node[anchor=north west] {$\cdots$};
            \draw [fill=xdxdff] (6,0) circle (2.5pt);
            \draw [fill=ududff] (5,2) circle (2.5pt);
            \draw (4.8,2.5) node[anchor=north west] {$z$};
            \draw (3.8,-0.3) node[anchor=north west] {$1$};
            \draw (4.3,-0.3) node[anchor=north west] {$2$};
            \draw (5.8,-0.3) node[anchor=north west] {$r$};
            \draw (6.4,1) node[anchor=north west] {.};
        \end{scriptsize}
    \end{tikzpicture}
    \caption{The Hasse diagram of $M_r$.}
    \label{fig:M_r}
\end{figure}
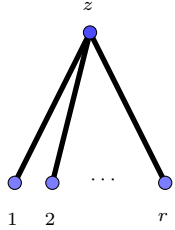
\end{example}

\begin{proposition}
\label{bin-PF}
    Let $P$ be a binomial poset of rank $N \in \NN \cup \{\infty\}$,  with factorial function $B$. Then $P$ is $\mathrm{TN}$ if and only if  the matrix $(1/B(n-k))_{n,k=0}^N$  is $\mathrm{TN}$, where $(1/B)(k) := 0$ if $k<0$.

    Hence if $N=\infty$, then  $P$ is $\mathrm{TN}$ if and only if $\{1/B(n)\}_{n=0}^\infty$is a P\'olya frequency sequence.  
\end{proposition}

\begin{proof}
    The numbers $r_{n,k}$ are given by 
    $$
        r_{n,k}= \frac {B(n)}{B(k) B(n-k)},  
    $$ see \cite{ehrenborg1995sheffer}. 
    Since total nonnegativity is closed under multiplying by diagonal matrices with positive entries on the diagonals, it follows that  $(r_{n,k})_{n,k=0}^\infty$ is $\mathrm{TN}$ if and only if $(1/B(n-k))_{n,k=0}^N$  is $\mathrm{TN}$. 

\end{proof}

\begin{example}\label{boolex}
    The sequence $\{1/n!\}_{n=0}^\infty$ is a P\'olya frequency sequence, since
    $$
        \sum_{n=0}^\infty \frac{1}{n!} x^n = e^x. 
    $$
    Hence, the infinite Boolean algebra is a $\mathrm{TN}$-poset, and the matrix $\left( \binom n k \right)_{n,k=0}^\infty$ is $\mathrm{TN}$.  From \eqref{pnchain} it follows that 
    $$
    p_n(t) = \sum_{k=1}^n k! S(n,k) t^k
    $$
    as claimed in Example \ref{binomet}. 
\end{example}

\subsection{Chain polynomials of (weak) upper homogeneous posets}  Next we will apply Theorem~\ref{maint} to chain enumeration in upper homogeneous posets, a class of posets recently introduced by Stanley \cite{stanley2020upho,stanley2024theorems}. This will give a combinatorial interpretation for the polynomials in Theorem~\ref{PFT} for all integer P\'olya frequency sequencies $(a_i)_{i=0}^\infty$ for which $a_0=1$. A poset $P$ is called \emph{upper homogeneous} (upho) if it is graded, has a least element $\zero$, has finitely many elements of each rank, and is such that each {principal} upper order ideal is isomorphic to the poset itself. For example, $\NN^n$ is an upho poset. 

We take the opportunity to generalize this construction. Recall the definition of the quasi-rank function $\rho$ associated to a poset. 
 A poset $P$ is called \emph{weak upper homogeneous} (wupho) if it 
 has finitely many elements of each quasi-rank, and if for {all} $x,y \in P$ with $\rho(x)=\rho(y)$, 
 $$
 |\{z  \in P : x \leq z \mbox{ and } \rho(z)=n \}| =  |\{z  \in P : y \leq z \mbox{ and } \rho(z)=n \}|, 
 $$
for each $n \geq \rho(x)$. Clearly any upho poset is wupho. 

 If $P$ is wupho, define a matrix $\overline{R}=\overline{R}(P)= (\overline{r}_{n,k}(P))_{n,k=0}^N$ by 
$$
\overline{r}_{n,k}(P) =   |\{z  \in P : x \leq z \mbox{ and } \rho(z)=n \}|,
$$
where $x \in P$ is a fixed element of rank $k$. 
 

\begin{proposition}
Let $P$ be a wupho poset, and let $n$ be a positive integer. Then the $n$th chain polynomial of $\overline{R}(P)$, see \eqref{polreceq}, is equal to 
\begin{equation}\label{pnupho}
 p_n(t)=\sum_{j \geq 1} |\{ x=x_0<x_1<\cdots<x_{j-1}<x_j : \rho(x_j) =n\}|t^j,
 \end{equation}
 where $x$ is a fixed minimal element of $P$. We call $p_n(t)$ the $n$th \emph{chain polynomial} of $P$. 
\end{proposition}

\begin{proof}
Let $f_0(t)=1$ and let $f_n(t)$, $n\geq 1$, be the polynomial on the right hand side of \eqref{pnupho}. Then, by conditioning on $\rho(x_{j-1})=k$, 
$$
f_n(t) = t\sum_{k=0}^{n-1}\overline{r}_{n,k}(P) f_k(t),
$$
and hence $\{f_n(t)\}_n$ satisfies the same recursion as $\{p_n(t)\}_n$. 
\end{proof}

\begin{definition}
We say that a wupho poset $P$ is $\overline{\mathrm{TN}}$ if the matrix $\overline{R}(P)$ is totally nonnegative. 
\end{definition}

Hence an upho poset is $\overline{\mathrm{TN}}$ if and only if it is Schur positive as defined in \cite{fu2024monoid}. In this case $\overline{R}(P)$ is a 
Toeplitz matrix $\overline{R}(P)= (r_{n-k}(P))_{n,k=0}^\infty$, where $r_n(P)= |\{x \in P: \rho(x)=n\}|$. 

Theorem~\ref{maint} immediately gives:

\begin{theorem}\label{uphot}
Let $P$ be a $\overline{\mathrm{TN}}$ wupho poset. Then, for each $n \geq 0$, all zeros of the $n$th chain polynomial $p_n(t)$ are real and located in the interval $[-1,0]$. Moreover $p_n(t) \prec p_{n+1}(t)$. 
\end{theorem}

Hopkins \cite{hopkins2024upho} constructs upho lattices from certain supersolvable lattices. These are all Schur positive. 

It was proved in \cite{fu2024monoid} that any P\'olya frequency sequence $(r_n)_{n=0}^\infty$ of positive  integers such that $r_0=1$ is the rank sequence $(r_n(P))_{n=0}^\infty$ of an upho poset $P$. Does this extend to wupho posets? 
\begin{problem}\label{TNTUTNP}
Suppose $\overline{R}= (\overline{r}_{n,k})_{n,k=0}^N$ is lower {unitriangular matrix} with nonnegative integer entries. If $\overline{R}$ is $\TN$ and $\overline{r}_{n, k} \geq \overline{r}_{n, k+1}$ for all $0\leq k <n\leq N$, then does there exist a wupho poset $P$ {such that} $\overline{R}(P)=\overline{R}$?
\end{problem}
The analog of Problem \ref{TNTUTNP} for quasi-rank uniform posets is true (even without the $\TN$-assumption) by Theorem \ref{TNTTNP}.


%
%
%

\section{Rank uniform posets with real-rooted chain polynomials}
\label{section-subdivision}

For simplicial complexes, it is often convenient to study the $h$-vector instead of the $f$-vector of the simplicial complex. Let $\Delta$ be an abstract simplicial complex (we consider the empty set a face of $\Delta$). The $f$-\emph{polynomial} of $\Delta$ is defined by 
$$
f_\Delta(t) = \sum_{S \in \Delta}t^{\# S}.
$$
If $\Delta$ has dimension $n-1$, then the $h$-\emph{vector} of $\Delta$ may be defined as the sequence $(h_0(\Delta), h_1(\Delta), \ldots, h_n(\Delta))$ for which 
$$
    f_\Delta(t) = \sum_{k=0}^n h_k(\Delta) t^k(1+t)^{n-k}.
$$
Hence the polynomials $t^k(1+t)^{n-k}$ play a crucial role in the study of chain polynomials of simplicial complexes. Brenti and Welker \cite{brenti2008f} proved that if the $h$-vector of a simplicial complex $\Delta$ is nonnegative, then the chain polynomial of $\Delta$ has only real zeros. Notice that the polynomials $t^k(1+t)^{n-k}$ are the polynomials $R_{n,k}(t)$ associated to the $\mathrm{TN}$ matrix $\left( \binom n k \right)_{n,k=0}^\infty$, which is the matrix $R(\mathbb{B})$ associated to the $\mathrm{TN}$-poset $\mathbb{B}$. In this section we will prove a vast generalization of Brenti and Welker's result, and in the process define natural analogs of $h$-vectors for all $\mathrm{TN}$-posets.

In this section all posets are assumed to have a least element $\zero$. 

\begin{definition}
Let $P$ be a quasi-rank uniform $\mathrm{TN}$-poset and let $(R_{n,k}(t))_{0\leq k \leq n \leq N}$ be the array of polynomials associated to the matrix $R= R(P)$, see Definition \ref{resolv} and Remark~\ref{thethe}.
Let   $Q$ be a quasi-rank uniform  poset of rank $r$.
\begin{itemize}
\item $Q$ is a \emph{$P$-poset} if for each $x \in Q$ there is a $y \in P$ such that  
$\{z \in Q : z \leq x\}$ is isomorphic to  $\{z \in P : z \leq y\}$. 
\item $Q$ is a \emph{combinatorial $P$-poset} if for each $y \in Q$, 
\begin{equation}\label{QaR}
\sum_{x \leq_Q y} t^{\nu(x)}= R_{m,0}(t), \ \ \ \ m=\nu(y),
\end{equation}
where $\nu$ is the quasi-rank function of $Q$.
\item A combinatorial $P$-poset $Q$ of rank $r$ is said to be \emph{$P$-positive} if its
rank generating polynomial,
$$
f_{Q}(t)= \sum_{x \in Q} t^{\nu(x)},
$$  has a nonnegative expansion in the polynomials $\{R_{r,k}(t)\}_{k=0}^{r}$, i.e., 
$$
f_{Q}(t) = \sum_{k=0}^r h_k(Q) \cdot R_{r,k}(t), 
$$
where $h_k(Q) \geq 0$ for all $0 \leq k \leq r$. The  \emph{$h$-vector} of $Q$ is the  vector $(h_0(Q), h_1(Q), \ldots, h_r(Q))$.  
\end{itemize}
\end{definition}
\begin{remark}
Clearly any $P$-poset is a combinatorial $P$-poset. 
Notice that a combinatorial $P$-poset may fail to have an $h$-vector, the reason being that $\{R_{r,k}(t)\}_{k=0}^r$ does not have to be a basis for the space of polynomials of degree at most $r$. However, for the $\mathrm{TN}$-posets considered in this paper it will be a basis. 
\end{remark}
\begin{example}
Examples of $\mathbb{B}$-positive posets are boolean cell complexes with nonnegative $h$-vectors, for example Cohen-Macaulay simplicial complexes.  
\end{example}

If $P$ is a finite poset, let $\hat{P}$ be the poset obtained by adding a new largest element $\one$ to $P$. If $P$ is quasi-rank uniform, then so is $\hat{P}$. 

\begin{theorem}\label{TN-Ppos}
Suppose $P$ is a $\mathrm{TN}$-poset, and that $Q$ is a (combinatorial) $P$-poset. Then $Q$ is $P$-positive if and only if $\hat{Q}$ is $\mathrm{TN}$. 
\end{theorem}

\begin{proof}
Suppose $\hat{Q}$ is $\mathrm{TN}$ of rank $n+1$. Then the matrix $R(\hat{Q})$ is resolvable by Theorem~\ref{eqcon}. Let $\{\hat{R}_{k,j}(t)\}$ be the polynomials that resolve $R(\hat{Q})$, and $\{R_{k,j}(t)\}_{k,j}$ the polynomials that resolve $R(P)$. Then $\hat{R}_{k,j}(t)= R_{k,j}(t)$ whenever $k \leq n$, and 
$$
f_{Q}(t)= \hat{R}_{n+1,0}(t)- t^{n+1}= \sum_{j=0}^n \lambda_{n,j}(\hat{Q}) \cdot R_{n,j}(t),
$$
by \eqref{r-sum}. Hence $Q$ is $P$-positive. 

Conversely if $Q$ is $P$-positive, then 
$$
f_{Q}(t)= \sum_{j=0}^n h_j \cdot R_{n,j}(t), 
$$
for nonnegative numbers $h_j$. Define $\lambda_{k,j}(\hat{Q}) = \lambda_{k,j}(P)$ for $0\leq j \leq k \leq n-1$, and $\lambda_{n,j}(\hat{Q})=h_j$ for $0 \leq j \leq n$. Define further $\hat{R}_{k,j}(t)=R_{k,j}(t)$ for $0\leq j \leq k \leq n$, and 
$$
 \hat{R}_{n+1,k}(t)= t^{n+1}+ \sum_{j \geq k} h_jR_{n,j}(t), \ \ \ \ 0\leq k \leq n.
$$
It follows that $R(\hat{Q})$ is resolvable. 
\end{proof}

\begin{example}
Suppose $P$ is a quasi-rank uniform $\mathrm{TN}$-poset and $y \in P$ has quasi-rank $n$. For $d \leq n$ consider the poset 
$$
Q=\{x \leq y : \rho(x) \leq d\}. 
$$
Then $Q$ is $P$-positive by Theorem~\ref{TN-Ppos} and Proposition~\ref{propranksel}.
\end{example}

The \emph{chain polynomial} of a finite poset $Q$ is the generating polynomial
$$
c_Q(t)= \sum_{C}t^{|C|}
$$
of the collection of all chains $C$ in $Q$. Notice that if $Q$ is quasi-rank uniform of rank $n$, and has a least element $\zero$, then 
$$
c_Q(t)= t^{-1}(1+t) p_{n+1}(t),
$$
where $p_{n+1}(t)$ is the $(n+1)$st chain polynomial of $\hat{Q}$. 

Theorem~\ref{rposp} below is a vast generalization of the theorem of Brenti and Welker \cite{brenti2008f}, who proved  the analog of Theorem~\ref{rposp} for the case when $P= \mathbb{B}$. 
\begin{theorem}
\label{rposp}
Let $P$ be a $\mathrm{TN}$-poset with a least element $\zero$. 
    All zeros of the chain polynomial of any $P$-positive poset $Q$ are located in the interval $[-1,0]$.
\end{theorem}

\begin{proof}
    Follows from Theorems \ref{TN-poset-chain} and \ref{TN-Ppos}.
\end{proof}

The next theorem is an immediate consequence of Theorems \ref{propranksel} and \ref{TN-Ppos}. 

\begin{theorem}
\label{rpossel}
Let $P$ be a $\mathrm{TN}$-poset with  a least element $\zero$, and let $S$ be an infinite set of nonnegative integers containing $0$. If $Q$ is a  $P$-positive poset, then $Q_S$ is $P_S$-positive. 
\end{theorem}
Hence if $P$ is a quasi-rank uniform $\mathrm{TN}$-poset with a least element $\zero$ and $Q$ is a $P$-positive poset, then the chain polynomial of any rank-selected {subposet} of $Q$ is real-rooted. 

\begin{example}
For the case when $P = \mathbb{B}$, Theorems \ref{rposp} and \ref{rpossel}  imply that if $Q$ is a boolean poset (or a simplicial complex) with nonnegative $h$-vector, then the chain polynomial of any rank selected {subposet} of $Q$ is real-rooted. This was first proved in \cite{athanasiadis2023two}.  
\end{example}

\subsection{Shellings of complexes and posets} 
There are several different definitions of shellings in the literature. The most common, the one for (abstract) simplicial complexes, reads as follows. Recall that an \emph{abstract simplicial complex} on a finite set $V$ is a collection $\Delta$ of subsets (called \emph{faces}) of $V$ which is closed under containment and contains $\{v\}$ for each $v \in V$. The \emph{dimension} of a face $S$ is  the number of elements in $S$ minus one, and the dimension of $\Delta$ is defined as the maximum dimension of a face of $\Delta$.   A simplicial complex is \emph{pure} if all maximal faces have the same dimension. Let $\Delta$ be a pure simplicial complex of dimension $d\geq 0$. An ordering $F_1, \ldots, F_m$ of the facets of $\Delta$ is a shelling if for each $j >1$, 
$$
\langle F_j \rangle \cap \bigcup_{i<j} \langle F_j \rangle 
$$
is pure of dimension $d-1$. If a {pure} simplicial complex has a shelling, then its $h$-vector is nonnegative. Can we generalize this to the notion of $h$-vector given in this paper? For this we should use suitable versions of shellings for the posets considered. 

If $P$ is a finite poset, then the \emph{order complex} $\Delta(P)$ is defined as the simplicial complex of all chains in $P$. Hence one may define shellability of posets via its order complex. Bj\"orner and Wachs \cite{BjornerCM, Bjorner-Wachs} considered versions of shellability of posets $P$ that imply shellability of $\Delta(P)$, namely EL-shellable and CL-shellable posets. For face-posets of polytopal complexes or regular CW-complexes the following notion of shellings, called \emph{C-shellings} in \cite{Stanley-94}, is often used. A finite poset $P$ is \emph{pure} if all maximal chains have the same length. We call the maximal elements of $P$ for \emph{facets}. 
\begin{definition}\label{CL-def}
Let $P$ be a pure poset  with a least element $\zero$.  An ordering $x_1, x_2, \ldots, x_m$ of the facets of $P$ is a \emph{C-shelling} if $P$ has rank at most one, or the rank of $P$ is at least two, and 
\begin{itemize}
\item[(S1)] for each $j > 1$, the {subposet} 
$$
H_j = \langle x_j \rangle \cap \bigcup_{i<j} \langle x_i \rangle  
$$
of $\langle x_j \rangle$ is pure of rank $\rho(P)-1$, and 

\item[(S2)] for each $j$, $[\zero, x_j)=\{x \in P: x <x_j\}$ has a C-shelling. Moreover, if $j>1$ then $[\zero, x_j)$ has a C-shelling for which the facets of $H_j$
come first. 
\end{itemize}
\end{definition}
\begin{remark}\label{Cshell-Dshell}
That $P$ has a C-shelling  is equivalent to that the dual of $P \cup \{\one\}$, where $\one$ is a new greatest element,  has a so called \emph{recursive atom ordering}, see \cite{Bjorner-Wachs}. By \cite[Theorem 3.2]{Bjorner-Wachs} this is equivalent to that the dual of $P$ is CL-shellable. In particular C-shellability of $P$ implies shellability of $\Delta(P)$. 
\end{remark}

\begin{remark}\label{C-shell-shell}
Recall that a finite lattice $L$ is \emph{lower semimodular} if for all $x,y \in L$, if $x \vee y$ covers $x$ and $y$, then $x \wedge y$ is covered by $x$ and $y$.  If $P$ is a lower semimodular lattice and $x \in P$, then any ordering of the facets of $[\zero, x)$ is a shelling of $[\zero, x)$, because then 
$x_i \wedge x_j$ is covered by $x_i$ and $x_j$ for any facets $x_i$ and $x_j$ of $[\zero, x)$. Hence for posets in which all principal order ideals are lower semimodular lattices, then shellability is equivalent to that just (S1) above is satisfied. In particular for face posets of simplicial complexes, the notions of C-shellability and shellability agree.  
\end{remark}

\section{Subspace lattices and the critical problem of Crapo and Rota}
\label{sec-critical}
\newcommand{\LLL}{\mathcal{L}}
Let $\FFF_q$ be a finite field {with $q$ elements}, and let $\mathbb{B}_d(q)$ be the lattice of all subspaces of $\FFF_q^d$. Let further $\mathbb{B}(q)$ be the lattice of all finite dimensional subspaces of the vector space $\{ f \in \FFF_q^\NN: f(i)=0 \mbox{ for all but finitely many } i \}$. Then $\mathbb{B}(q)$ is a binomial poset, and if $y \in \mathbb{B}(q)$ has rank (dimension) $n$, then 
$$
|\{ x \leq y : \rho(x)=k \}| = {\binom n k}_{\! \! \! q},
$$
where ${\binom n k}_{\! q}$ is a $q$-binomial number \cite[Chapter 1.7]{stanley2011enumerative}. Rota introduced \emph{$q$-posets} in \cite{Rota71}, and Alder studied them further in~\cite{alder2010q}. In our terminology $q$-posets are the same as  $\mathbb{B}(q)$-posets.

Let $P$ be a $q$-poset. 
If $\rho(y)=n$, then 
$$
    R^q_n(t)= \sum_{x \leq y} t^{\rho(x)} = \sum_{k=0}^n {\binom n k}_{\! \! \! q} t^k. 
$$
These polynomials are known as the Rogers-Szeg\H{o} polynomials \cite[Chapter 3]{szego1926beitrag}; they satisfy the recursion ${R}^q_0 (t) =1$, and 
\begin{equation}\label{RoSz-eq}
  {R}_{n+1}^q(t) = t{R}_n^q(t)+ {R}_n^q(qt), \ \ \ \ \mbox{ for } n \geq 0. 
\end{equation}
Let $\alpha \colon \RR[t] \to \RR[t]$ be the diagonal linear operator defined by 
$
    \alpha(f)(t) = f(qt). 
$
 {Then by the recursion \eqref{RoSz-eq},}
$
    {R}_n^q(t) = (t+\alpha)^n 1, 
$
which by Lemma~\ref{eqcon} proves the well known fact that the matrix 
$\left({\binom n k}_{\!  q}\right)_{n,k=0}^\infty$
is $\mathrm{TN}$ for $q>0$. 

Recall that 
$
    R^q_{n,k}(t)= (t+\alpha)^{n-k}t^k 1.   
$
Lemma~\ref{eqcon} then implies the following recursion. 
\begin{proposition}
\label{SSRrec}
    If $0\leq k < n+1$, then 
    \begin{equation}\label{RS1}
        R^q_{n+1,k+1}(t)=R^q_{n+1,k}(t) - q^kR^q_{n,k}(t). 
    \end{equation}
\end{proposition}

\begin{remark}
It is not hard to prove using \eqref{RS1} that $R^q_{n,k}(t)= q^{k(n-k)} t^k R^q_{n-k}(q^{-k}t)$, but since we will not use it here we leave the proof to the reader. 
\end{remark}

Denote by $\EE_q$, the subdivision operator associated to $({\binom n k}_{\! q})_{n,k=0}^\infty$, see Definition~\ref{subop}. Notice that the definition of $R^q_{n,k}(t)$ makes sense for any nonnegative real number $q$. 
\begin{theorem}
\label{mainssl}
    Let $q$ be a nonnegative real number. The sequence $\{R^q_{n}(t)\}_{n=0}^\infty$ is resolvable. In particular, for each $n \geq 0$, $\{\EE_q(R^q_{n,k}(t))\}_{k=0}^n$ is an interlacing sequence of polynomials whose zeros lie in $[-1,0]$. 
\end{theorem}

\begin{proof}
 The theorem follows from Theorem~\ref{mainUTP} since  $({\binom n k}_{\! q})_{n,k=0}^\infty$ is $\mathrm{TN}$. 
\end{proof}

\subsection{The critical problem and expansions of characteristic polynomials}
To apply Definition \ref{resolv} and Theorem~\ref{mainssl}, we want to have an understanding for when $q$-posets are $\mathbb{B}(q)$-positive. It turns out that this is closely related to positivity questions regarding characteristic {polynomials} of hyperplane arrangements in $\FFF_q^n$, and the \emph{Critical Problem} of Crapo and Rota \cite[Chapter 16]{crapo1970foundations}. 

 Let $\HH = \{H_1,\ldots, H_m\}$ be a collection (list) of hyperplanes in $\mathbb{B}_n(q)$.  
    Recall that the \emph{characteristic polynomial} of $\HH$ is 
$$
    \chi_\HH(t) = \sum_{A \subseteq \{1,\ldots,m\}} (-1)^{|A|}  t^{\dim\left( \cap_{i \in A} H_i\right)}.
$$

For $0\leq k \leq n$ and $q \in \RR$, define polynomials $\chi^q_{n,k}(t)$ by $\chi^q_{n,0}(t)=t^n$, and 
$$
\chi^q_{n,k}(t)= t^{n-k}(t-1)(t-q)\cdots (t-q^{k-1}),
$$
otherwise. The reason why these polynomials are important for us is the next lemma.  Define a bijective linear operator $R^q : \RR[t] \to \RR[t]$ by $R^q(t^n)=R_n^q(t)$ for each $n \in \NN$.  
\begin{lemma}\label{chitor}
For each $0 \leq k \leq n$, 
$
R^q(\chi_{n,k}^q)(t) =R^q_{n,k}(t).
$
\end{lemma}
\begin{proof}
A direct computation proves the recursion 
$$
\chi^q_{n,k+1}(t)=\chi^q_{n,k}(t)-q^k\chi^q_{n-1,k}(t), 
$$
from which the lemma follows from Proposition~\ref{SSRrec} by induction on $k$.
\end{proof}
In light of Definition \ref{TN-poset} and Lemma~\ref{chitor} one might ask for which hyperplane arrangements $\HH$ in $\FFF_q^n$, the characteristic polynomial $\chi_\HH (t)$ has a nonnegative expansion in the polynomials $\{\chi^q_{n,k}(t)\}_{k=0}^n$. We will prove that all do.  In particular the following theorem holds for all chromatic polynomials of graphs. 
\begin{theorem}\label{posexpchi}Let $\HH$ be a hyperplane arrangement in $\FFF_q^n$. The characteristic polynomial of $\HH$ has a unique expansion 
$$
\chi_\HH(t)= \sum_{k=0}^{n} \theta_k(\HH) \cdot \chi^q_{n,k}(t), 
$$
where $\theta_k(\HH) \geq 0$ for all $k$, and $\theta_0(\HH)+\theta_1(\HH)+\cdots+\theta_{n}(\HH)=1$. 
\end{theorem}
Theorem~\ref{posexpchi} has the following consequence.
\begin{corollary}\label{ch-cor}
The convex hull of the set of all characteristic polynomials of hyperplane arrangements in $\FFF_q^n$ is equal to the simplex 
$$
\left\{ \sum_{k=0}^{n} \theta_k \cdot \chi^q_{n,k}(t) : \theta_k \geq 0 \mbox{ for all } k, \mbox{ and } \theta_0 + \theta_1+ \cdots +\theta_{n}=1\right\}. 
$$
\end{corollary}

\begin{proof}
By Theorem~\ref{posexpchi} the convex hull of the characteristic polynomials is a subset of the simplex. The converse follows by the fact that if $U$ is a subspace of $\FFF_q^n$ of dimension $k$, then the characteristic polynomial of the arrangement 
$$
\HH= \{H : \mbox{ the normal of $H$ is in $U$}\}
$$ 
is equal to $\chi^q_{n,k}(t)$. This may be proved using  Theorem~\ref{crita}. 
\end{proof}

To prove Theorem~\ref{posexpchi}  we recall the Critical Problem of Crapo and Rota. Recall that the \emph{matroid} associated to the hyperplane arrangement $\HH=\{H_1,\ldots, H_m\}$ is the (ordered) set $\MM= \{ n_1,\ldots, n_m\} \subseteq \FFF_q^n \setminus \{(0)\}$, where $n_i$ is a normal vector of $H_i$ for each $i$. We write $\chi_\MM(t)$ for $\chi_\HH(t)$. 

\begin{theorem}\cite[Thm 1 in Ch. 16]{crapo1970foundations}
\label{crita}
    Let $\HH$ be a hyperplane arrangement in $\FFF_q^n$, and let $m \in \NN$. Then 
    $$
        \chi_{\MM}(q^m) = |\{ \varphi \in \Hom(\FFF_q^n, \FFF_q^m) : \ker(\varphi)\cap \MM= \varnothing\}|, 
    $$
    Moreover if $F$ is a flat of $\MM$, then 
    \begin{equation}\label{q-flats}
        \chi_{\MM/F}(q^m) = |\{ \varphi \in \Hom(\FFF_q^n, \FFF_q^m) : \ker(\varphi)\cap \MM= F\}|, 
    \end{equation}
    where $\MM/F$ is considered a subset of $\FFF_q^n / \langle F\rangle \cong \FFF_q^{n-r(F)}$. 
\end{theorem}
For $n \in \NN$, let $\mathsf{D}_n \colon \RR[t] \to \RR[t]$ be the linear operator defined by 
$$
(\mathsf{D}_n f)(t) = f(qt)-q^nf(t),  
$$
and notice that  $\mathsf{D}_n \chi^q_{n,0}(t) \equiv 0$, and 
$$
    \mathsf{D}_n \chi^q_{n,k}(t)= q^{n-1}(q^k-1)\chi^q_{n-1,k-1}(t), \ \ \ \mbox{ for } 1\leq k \leq n. 
$$
Also, 
$$
\chi_\HH(1)= \theta_0(\HH)= \begin{cases}
1 &\mbox{ if } \HH = \varnothing, \mbox{ and } \\
0 &\mbox{ otherwise.}
\end{cases}
$$
Hence to prove that $\chi_\HH(t)$ has a nonnegative expansion in the polynomials  $\chi^q_{n,k}(t)$ it suffices to prove that $\mathsf{D}_n \chi_\HH(t)$ may be expressed as a nonnegative sum of characteristic polynomials $\chi_{\HH'}(t)$ of hyperplane arrangements $\HH'$ in $\FFF_q^{n-1}$. Hence we want to prove: 
\begin{theorem}\label{fundn}Let $\MM \subseteq \FFF_q^n \setminus \{(0)\}$ be nonempty. Then 
$$
\mathsf{D}_n\chi_\MM(t) = \sum_{\MM'} \lambda(\MM') \cdot \chi_{\MM'}(t),
$$
where $\MM' \subseteq \FFF_q^{n-1} \setminus \{(0)\}$ and $\lambda(\MM') \geq 0$. 
\end{theorem}

\begin{lemma}\label{Dn-form}Let $\MM \subseteq \FFF_q^n \setminus \{(0)\}$ be nonempty. 
Then 
$$
\mathsf{D}_n\chi_\MM(t)= \sum_{F} \chi_{\MM/F}(q)\cdot \left( \chi_F(t) -\chi_\MM(t)\right),
$$
where the sum is over all flats $F$ of $\MM$. 
\end{lemma}
\begin{proof}
Let $C_\MM(n,m)= \{ \varphi \in \Hom(\FFF_q^n, \FFF_q^m) : \ker(\varphi)\cap \MM= \varnothing\}$ and consider the injective map 
$$
\Phi : C_\MM(n,m)\times \Hom(\FFF_q^n , \FFF_q) \to C_\MM(n,m+1)
$$
given by 
$$
\Phi(\varphi, \lambda)(x)= \varphi(x)+ \lambda(x)e_{m+1},
$$
where $e_{m+1}$ is the $(m+1)$st standard basis vector of $\FFF_q^{m+1}$. Then the number of elements in $C_\MM(n,m+1)$ which are not in the image of $\Phi$ is 
$$
  \chi_\MM(q^{m+1})- q^n \chi_\MM(q^{m}) = (D_n\chi_\MM)(q^m),
$$
by Theorem~\ref{crita}. 
Hence $D_n\chi_\MM(q^m)$ is equal to the number of pairs $(\varphi,\lambda)$, where $\varphi \in \Hom(\FFF_q^n, \FFF_q^m)$ and $\lambda \in \Hom(\FFF_q^n, \FFF_q)$ are such that 
\begin{itemize}
\item $\ker(\varphi) \cap \MM \neq \varnothing$, and 
\item $\ker(\lambda) \cap \ker(\varphi) \cap \MM = \varnothing$.  
\end{itemize} 
Let us count the number of such $(\varphi,\lambda)$. 
Suppose $\ker(\lambda) \cap \MM= F \neq  \MM$. The number of such $\lambda$ is $\chi_{\MM/F}(q)$ by Theorem~\ref{crita}. For $\varphi$ we have  $\varnothing \neq \ker(\varphi)\cap \MM \subseteq \MM \setminus F$, {since $\ker(\varphi) \cap F = \varnothing$}. Thus the number of such $\varphi$ is $\chi_{F}(q^m)- \chi_\MM(q^m)$, and hence 
$$
\mathsf{D}_n\chi_\MM(q^m)= \sum_{F < \MM} \chi_{\MM/F}(q)\cdot \left( \chi_F(q^m) -\chi_\MM(q^m)\right),
$$
from which the lemma follows since the identity holds for all positive $m$. 
\end{proof}

\begin{lemma}\label{dccon}
If $A \subseteq B \subseteq \FFF_q^n \setminus \{(0)\}$, then 
$$
\chi_A(t)-\chi_B(t) = \sum_{C} \chi_C(t),
$$
where the sum is over a collection of $C \subseteq \FFF_q^{n-1} \setminus \{(0)\}$. 
\end{lemma}

\begin{proof}
Let $e \in B \setminus A$. By deletion/contraction, 
$$
\chi_B(t)= \chi_{B \setminus e}(t)-\chi_{B/e}(t),
$$
where $\chi_{B/e}(t)$ is the characteristic polynomial of an arrangement in $\FFF_q^{n-1} \setminus \{(0)\}$. Hence 
$$
\chi_A(t)-\chi_B(t) = (\chi_A(t)-\chi_{B \setminus e}(t))+\chi_{B/e}(t),
$$
from which the lemma follows by iterating. 
\end{proof}

Theorem~\ref{fundn} now follows from Lemma~\ref{Dn-form} and Lemma~\ref{dccon}. 

\subsection{Shellable  \texorpdfstring{$q$}--posets} 
\newcommand{\UU}{\mathcal{U}}
 Alder \cite{alder2010q} considered shellings of $q$-posets. His notion of shellings is orderings of the facets which satisfy (S1). Since $\mathbb{B}(q)$ is modular Alder's notion is equivalent to C-shellings by Remark \ref{C-shell-shell}. In particular, by Remark \ref{Cshell-Dshell}, if a $q$-poset $P$ is shellable according to Alder, then $\Delta(P)$ is also shellable. This was conjectured in \cite[Section 5]{Ghorpade2}. 
In particular  if $H_1, \ldots, H_m$ are distinct hyperplanes of $\mathbb{B}_{n+1}(q)$, then 
$P= \langle H_1\rangle \cup \cdots \cup \langle H_{\ell} \rangle$ is shellable, and any ordering of the hyperplanes is a shelling.

\begin{theorem}\label{expU}
If $P$ is a C-shellable $q$-poset, then $P$ is $\mathbb{B}(q)$-positive. 
In particular the chain polynomials of $P$ and its rank selected {subposet}s are real-rooted. 
\end{theorem}

\begin{proof}
Suppose $P$ has rank $n$, and let $F_1,F_2,\ldots, F_m$ be a shelling of $P$. Let $P_k= \langle F_1\rangle \cup \cdots \cup \langle F_{k} \rangle$. Then $f_{P_1}(t)=R^q_n(t)$. With the notation above, we may write $P_k$ as the disjoint union 
$
P_k= P_{k-1}\cup Q_k,
$ 
where 
$$
Q_k= \langle F_k \rangle \setminus \big(\langle H_1\rangle \cup \cdots \cup \langle H_{\ell} \rangle\big),
$$
and where $H_i$ is a facet of {$P_{k-1} \cap F_k$} for each $i$. We prove by induction on $k$ that $P_k$ is $\mathbb{B}(q)$-positive, the case when $k=1$ being obvious. Let $k>1$. 
Since $\langle F_k \rangle \cong \mathbb{B}_n(q)$ we may assume $F_k=\mathbb{B}_n(q)$ in the following computation. By Inclusion-Exclusion,
$$
\sum_{x \in Q_k} t^{\dim(x)}= \sum_{{S \subseteq [\ell]}} (-1)^{|S|}\left(\sum_{x \in \cap_{i \in S}\langle H_i \rangle }t^{\dim(x)} \right)= R^q_n(\chi_\HH)(t). 
$$
Hence 
$$
f_{P_k}(t)=f_{P_{k-1}}(t)+ R^q_n(\chi_\HH)(t),
$$
and the theorem follows from Theorem~\ref{posexpchi},  Lemma~\ref{chitor} and Theorem~\ref{rpossel}. 
\end{proof}
In \cite{alder2010q}, Alder raised the question of finding a suitable notion of $h$-vectors for $q$-posets $P$. In particular it should have the property that the $h$-vector is nonnegative whenever $P$ is shellable, and it should specialize to the common $h$-vector for simplicial complexes when $q=1$. Theorem~\ref{expU} offers such a notion: if $P$ is a $q$-poset of rank $n$, then the $h$-vector of $P$ is defined as the unique vector $(h_0(P), h_1(P), \ldots, h_n(P))$ for which 
$$
f_P(t) = \sum_{k=0}^n h_k(P) \cdot R^q_{n,k}(t).
$$

An interesting class of shellable $q$-posets are related to $q$-matroids \cite{jurrius2018defining}. 

\begin{definition}
A $q$-\emph{matroid} on $\mathbb{B}_n(q)$ is a function $\varphi : \mathbb{B}_n(q) \to \NN$ satisfying 
\begin{enumerate}
    \item $\varphi(x) \leq \dim(x)$, for all $x \in \mathbb{B}_n(q)$, and 
    \item $\varphi(x) \leq \varphi(y)$, for all $x \leq y \in \mathbb{B}_n(q)$, and
    \item $\varphi(x\vee y)+\varphi(x \wedge y) \leq \varphi(x) + \varphi(y)$, for all $x,y \in \mathbb{B}_n(q)$. 
\end{enumerate}
\end{definition}
For example if $0 \leq r \leq n$, then $\varphi : \mathbb{B}_n(q) \to \NN$ defined by 
$$
\varphi(x)= \min(\dim(x),r)
$$
is a $q$-matroid.

The \emph{set of independent spaces} of a $q$-matroid on $\mathbb{B}_n(q)$ is 
$$
P(\varphi)= \{x \in \mathbb{B}_n(q) : \varphi(x)= \dim(x)\}. 
$$
This poset is a pure order ideal of $\mathbb{B}_n(q)$, and hence a pure $q$-poset. 
\begin{theorem}[\cite{Ghorpade}]
\label{qmth}
    If $\varphi$ is a $q$-matroid, then $P(\varphi)$ is shellable. 
\end{theorem}
The rank of a $q$-matroid $\varphi$ on $\mathbb{B}_n(q)$ is $\max\{ \varphi(x) : x \in \mathbb{B}_n(q)\}$. 
From Theorems \ref{expU}, \ref{qmth} and \ref{rpossel} we deduce the following theorem.
\begin{theorem}\label{qmrr}
If $\varphi$ is a $q$-matroid, then $P(\varphi)$ is $\mathbb{B}(q)$-positive. 
In particular the chain polynomials of $P(\varphi)$ and its rank selected {subposet}s are  real-rooted. 
\end{theorem}

\section{\texorpdfstring{$r$}--cubical posets}
\label{sec-r-cubical}
In this section we study $r$-cubical lattices and posets, see \cite{ehrenborg1996r}. 
Let $r$ be a positive integer. Recall the definition of $r$-cubical lattice $\mathbf{C}_r$ in Example \ref{def-r-cube}. If $x \in \mathbf{C}_r$ has rank $n$, then we call the interval $[\zero, x]$ an $r$-\emph{cube} of rank $n$ and denote it by $C_{n,r}$. The rank generating polynomials of $\mathbf{C}_r$ are 
$
    R_n (t) = 1 + t(r+t)^{n-1}.
$
Define 
\begin{equation}
\label{rec-r-cube}
    R_{n,k}(t) = \begin{cases}
        R_n(t) &\mbox{ if } k = 0, \\
        (r-1+t)t^k(r+t)^{n-k-1} &\mbox{ if } 0 < k < n, \\
        t^n &\mbox{ if } k =n.
    \end{cases}\,  
\end{equation}

The next lemma proves that $\mathbf{C}_r$ is a $\mathrm{TN}$-poset. 

\begin{lemma}
\label{matrix-r-cube}
The matrix $R(\mathbf{C}_r)$ is totally nonnegative. Moreover 
 $$
 R_{n+1,k}(t)=R_{n+1,k+1}(t)+ \lambda_{n,k} R_{n,k}(t), \ \ \ \ 0 \leq k \leq n, 
 $$
    where
    $$
    \lambda_{n,k} = 
        \begin{cases}
            1 &\mbox{ if } k=0, \\
            r &\mbox{ if } 1 \leq k < n, \\
            r-1 &\mbox{ if } k = n.
        \end{cases}\,
    $$
\end{lemma}

\begin{proof}
It is straightforward to check that the polynomials $R_{n,k}(t)$ satisfy the stated recursion, and hence 
$R(\mathbf{C}_r)$ is resolvable. The lemma now follows from Theorem~\ref{eqcon}. 
\end{proof}

We call  $\mathbf{C}_r$-posets \emph{$r$-cubical posets}. A $2$-cubical poset is called \emph{cubical}.

In \cite{adin1996new} Adin introduced an $h$-vector for cubical posets. Let $f_P(t)$ be the rank generating polynomial of a cubical poset $P$ of rank $n$, i.e., 
$$
f_P(t)= \sum_{x \in P} t ^{\rho(x)}. 
$$ 
The (normalized) cubical $h$-\emph{vector} \cite{adin1996new, athanasiadis2021face, ehrenborg2000flags}  of $P$ may be defined as the vector  $(h_0(P), h_1(P), \ldots, h_n(P))$ for which 

\begin{equation}
\label{q-h-vector}
\begin{split}
   (1+t) h_P(t) = &(-1)^{n} f_P(-1)t^{n+1} + f_P(0) \left( 1 - \left( \frac{1-t}{2} \right)^{n} \right) + \\
                          &\left( \frac{1-t}{2} \right)^{n} f_P \left( \frac{2t}{1-t} \right),
\end{split}
\end{equation}
where $h_P(t)= \sum_{i=0}^nh_i(P)t^i$. 

\begin{theorem}\label{q-h-pos}
Let $P$ be cubical poset. Then $P$ has nonnegative cubical $h$-vector if and only if $P$ is $\mathbf{C}_2$-positive. 
\end{theorem}

\begin{proof}
Notice that \eqref{q-h-vector} defines a unique bijective linear operator 
$H : \RR_n[t] \to \RR_n[t]$ satisfying $H(f_P)(t) = h_P(t)$ for all {cubical posets} $P$ of rank $n$. It is straightforward to check that 
$$
H(R_{n,k})(t)= \begin{cases}
t^k/2 &\mbox{ if } 0<k<n,  \\
t^k &\mbox{ otherwise,}
\end{cases}
$$
from which the theorem follows. 
\end{proof}
Adin \cite{adin1996new} proved that shellable cubical complexes have nonnegative cubical $h$-vectors. Athanasiadis \cite{athanasiadis2021face} proved that the $f$-polynomial of the barycentric subdivision of any cubical complex with nonnegative cubical $h$-vector is real-rooted, thus confirming a conjecture of Brenti, Mohammadi and Welker \cite{brenti2008f}. An alternative proof of Athanasiadis' result now follows directly from {the following theorem.}
\begin{theorem}
If $P$ is a cubical poset with nonnegative cubical $h$-vector, then the chain polynomial of $P$ is real-rooted. Moreover if $S$ is a set of positive integers, then the chain polynomial of $P_S$ is real-rooted. 
\end{theorem}

\begin{proof}
Apply Theorems \ref{rposp}, \ref{rpossel} and \ref{q-h-pos}.  
\end{proof}

Lemma~\ref{matrix-r-cube} and Theorem~\ref{q-h-pos} suggest how to define a suitable $r$-cubical $h$-vector. 

 \begin{definition}
 \label{qh-def}
 Let $P$ be an $r$-cubical poset of rank $n$. The $r$-cubical $h$-vector of $P$ is the unique vector $(h_0(P),h_1(P), \ldots, h_n(P))$ for which 
 $$
 f_P(t) = \sum_{k=0}^n h_k(P)\cdot R_{n,k}(t),
 $$ 
 where $R_{n,k}(t)$ is defined by \eqref{rec-r-cube} for the matrix $R=R(\mathbf{C}_r)$. 
 \end{definition}
By the proof of Theorem~\ref{q-h-pos} we see that for $r=2$, the entries of the cubical $h$-vector \eqref{q-h-vector} and Definition \ref{qh-def}  agree up to a factor of $2$. 

\subsection{Shellings of $r$-cubical posets}
We want to prove that C-shellable $r$-cubical posets have nonnegative $h$-vectors, thus extending Adin's result. To do this we need to understand C-shelling of $r$-cubical posets better.

Suppose $P$ is an $r$-cubical poset. For each $y \in P \setminus \{\zero\}$, we identify $y$ with an element of $\mathbf{C}_r$ of the same rank.   If $x \prec y \in P$, let 
$
\gamma(x,y) 
$
be the index for which the coordinate of $x$ and $y$ differ (as elements of $\mathbf{C}_r$). 
For $x = (\xi_1,\xi_2,\ldots) \in \mathbf{C}_r \setminus \{\zero\}$, let $E(x) = \{ i : \xi_i = z\}$. 

The next lemma describes all C-shellable pure order ideals of $\mathbf{C}_r$ which are generated by elements in $[\zero, y]$, $y \in \mathbf{C}_r$, of co-rank one. 
\begin{lemma}\label{indshelemma}
Suppose $y$ is an element of $\mathbf{C}_r$ of rank $n+1$, where $n \geq 2$. 
Let $X$ be a set of elements that are covered {by} $y$.  Then there exists a C-shelling of $[\zero, y)$ where the elements of $X$ come first if and only if 
\begin{itemize} 
\item[(i)] there is  $x \in X$ for which $\gamma(x,y) \neq \gamma(x',y)$ for all other  $x' \in X$, or 
\item[(ii)] $\{ \gamma(x,y) : x \in X\} = E(y)$.
\end{itemize}
\end{lemma}
\begin{proof}
We first prove the sufficiency of (i) or (ii) by induction over $n \geq 2$. Suppose $n \geq 2$, and that $X$ satisfies (i) or (ii). Let $X' \subseteq X$ be a maximal subset of $X$ for which 
$\gamma(x',y) \neq \gamma(x,y)$ for all distinct $x$ and $x'$ in $X'$. Let further $L: x_1, \ldots, x_m$ be any ordering of the facets of $[\zero, y)$ satisfying 
\begin{enumerate}
\item the elements of $X'$ form an initial segment of $L$, 
\item the elements of $X$ form an initial segment of $L$, and 
\item if $X$ satisfies (i), then fix an element $x \in X$ which satisfies (i). Then the set of facets $x'$ of $[\zero, y)$ for which $\gamma(x',y)=\gamma(x,y)$ is required to form a final segment of $L$.
\end{enumerate}
We shall prove that $L$ is a C-shelling. Suppose $x$ and $x'$ are two distinct facets of $[\zero, y)$. If  $\gamma(x,y)=\gamma(x', y)$, then $\langle x \rangle \cap \langle x' \rangle = \{\zero \}$.  If $\gamma(x,y) \neq \gamma(x', y)$, then $\langle x \rangle \cap \langle x' \rangle =\langle w \rangle$, where $w= x \wedge x'$ is obtained from $x$ by replacing the $z$ in position $\gamma(x',y)$ to the entry in position $\gamma(x',y)$ of $x'$. Hence $\rho(w)= \rho(x)-1$. 
By this observation, and the construction of $L$, it follows that
$$
H_j = \langle x_j \rangle \cap \bigcup_{i<j} \langle x_i \rangle  
$$
is pure of rank $n-1$ for each $2 \leq j \leq \ell$. If $n=2$ this proves that $L$ is a C-shelling. 
Suppose $n>2$. The set of facets of 
$H_j$ is  
$$
X'' = \{\pi_j(x_i) : i<j \mbox{ and } \gamma(x_i,y) \neq \gamma(x_j,y) \}, 
$$
where $\pi_j(x_i)$ is obtained from $x_i$ by exchanging the $z$ in position $\gamma(x_j,y)$ to the coordinate in position $\gamma(x_j,y)$ of $x_j$. By construction, $X''$ satisfies (i) or (ii), and hence the sufficiency follows by induction. 
%
%
%

The necessity of (i) and (ii) follows from the following claim: let $n\geq 2$ be an integer. If $X$ is a set of facets of $[\zero, y)$, where $y \in \mathbf{C}_r$ has rank $n+1$, satisfying  
\begin{itemize}
\item[(I)] for each $x \in X$, there exists $x' \in X\setminus \{x\}$ for which $\gamma(x,y) = \gamma(x',y)$, and 
\item[(II)] $\{ \gamma(x,y) : x \in X\} \neq E(y)$, 
\end{itemize}
then there is no shelling of $[\zero, y)$ that starts with the elements of $X$. 

We prove the claim by induction over $n \geq 2$. If $n=2$, we may assume $y=(z,z)$ and that 
$\gamma(x,y)= 2$ for all $x \in X$, where $|X| \geq 2$. It follows that $\langle x \rangle \cap \langle x' \rangle = \{\zero\}$ for all distinct $x$ and $x'$ in $X$, which proves the case when $n=2$. 
Suppose $n >2$, and that $X$ satisfies (I) and (II). Suppose $x_1, x_2, \ldots, x_\ell$ is an ordering of $X$ which extends to a C-shelling of $[\zero, y)$. Then the facets of $H_\ell$ are precisely the elements 
$$
X''= \{\pi_\ell(x_i) : 1\leq i<\ell \mbox{ and } \gamma(x_i,y) \neq \gamma(x_\ell,y) \}.
$$
By construction $X''$ satisfies (I) and (II) for $y = x_\ell$, and hence the claim follows by induction. 
\end{proof}

From Lemma \ref{indshelemma} we deduce the following characterization of C-shellings of $r$-cubical posets. 

\begin{theorem}\label{r-cube-shelling}
Let $P$ be an $r$-cubical poset. Then a list $y_1, \ldots, y_m$ of the facets of $P$ is a C-shelling of $P$ if and only if for each $j >1$, the {subposet} 
$$
H_j = \langle y_j \rangle \cap \bigcup_{i<j} \langle y_i \rangle  
$$
of $\langle y_j \rangle$ is pure of rank $\rho(P)-1$,  and either
\begin{itemize} 
\item[(i)] there is a facet $x$ of $H_j$ for which $\gamma(x,y_j) \neq \gamma(x',y_j)$ for all other facets $x'$ of $H_j$, or 
\item[(ii)] for each $i \in E(y_j)$ there is a facet $x \in H_j$ for which 
$\gamma(x,y_j)=i$. 
\end{itemize}
\end{theorem}


\begin{theorem}
Let $P$ be a pure  $r$-cubical poset of rank $n$. If $P$ is C-shellable, then $P$ has a nonnegative $r$-cubical $h$-vector. In particular, $P$ is $\mathbf{C}_r$-positive. 
\end{theorem}

\begin{proof}
Let $y_1,\ldots, y_m$ be a C-shelling of $P$. We prove that the $r$-cubical $h$-vector of $P_j=\langle y_1\rangle \cup \cdots \cup \langle y_j\rangle$ is nonnegative by induction over $j$. If 
$j=1$, then $f_{P_j}(t)= R_{n,0}(t)$ and hence the $h$-vector is $(1,0,0,\ldots)$. 

Let $j \geq 2$. Then 
$$
f_{P_j}(t)= f_{P_{j-1}}(t)+ \sum_{x \in \langle y_j \rangle \setminus (P_{j-1} \cap \langle y_j \rangle)} t^{\rho(x)}. 
$$
By induction $f_{P_{j-1}}(t)$ has a nonnegative expansion in the polynomials $R_{n,k}(t)$. It remains to prove that the second summand, which we name $S_2(t)$, also does. For $i \in E(y_j)$, let  $A_i$ be the set of all $a \in [r]$ for which there is a {facet} of $P_{j-1} \cap \langle y_j \rangle$ with an $a$ in the $i$th coordinate.  Then 
$x \in \langle y_j \rangle \setminus (P_{j-1} \cap \langle y_j \rangle)$ if and only if the $i$th coordinate of $x$ is in  $\{z\} \cup [r] \setminus A_i$ for each $i \in E(y_j)$. Hence 
$$
S_2(t)= t\! \!\!\prod_{i \in E(y_j) } (t+ |[r]\setminus A_i|) = t\! \!\!\prod_{i \in E(y_j)} (t+ r-|A_i|). 
$$
By Theorem \ref{r-cube-shelling}, either $|A_i|=1$ for some $i$, or $|A_i|>0$ for all $i$.  If $|A_i|=1$ for some $i$, then we may write $S_2(t)= t(t+r-1)g(t)$, where $g(t)$ is a product of factors of the form $t+i$, where $0 \leq i \leq r$. Since 
\begin{equation}\label{rit}
t+i = \frac {r-i} r t + \frac i r (r+t),
\end{equation}
$g(t)$ has a nonnegative expansion in $t^j (r+t)^{n-2-j}$, and thus $S_2(t)$ has a nonnegative expansion in $R_{n,k}(t)$. 

Similarly, if $|A_i|>0$ for all $i$, then we may write $S_2(t)= t h(t)$, where $h(t)$ is a product of factors of the form $t+i$, where $0 \leq i \leq r-1$. By the same argument as before, it follows that 
we may write 
$$
t h(t) = \sum_{k=0}^{n-1} a_k t^{k+1} (r-1+t)^{n-1-k} = a_{n-1}t^{n} + (r-1+t)\sum_{k=0}^{n-2} a_k t^{k+1} (r-1+t)^{n-2-k}
$$
where $a_k \geq 0$ for each $k$. By \eqref{rit}, each term $t^{k+1} (r-1+t)^{n-2-k}$ may be expanded nonnegatively in $t^i(r+t)^{n-1-i}$, $ 1 \leq i \leq n-1$. Hence $S_2(t)$ has a nonnegative expansion in $R_{n,k}(t)$. This finishes the proof by induction. 

\end{proof}

\begin{corollary}
    If $P$ is C-shellable $r$-cubical poset, then chain polynomial of (any rank-selected {subposet} of) $P$ is real-rooted. 
\end{corollary}

\section{Dual partition lattices}
\label{sec-part}

Recall that the dual partition lattice of rank $n$, $\Pi_{n+1}'$, is the poset of all partitions of  $[n+1]$, with $\pi \leq \sigma$ if every block of $\sigma$ is contained in a block of $\pi$.
The \emph{infinite dual partition lattice} $\Pi'$ is the set of all partitions of  $\{1,2,\ldots \}$ into finitely many blocks, with $\pi \leq \sigma$ if every block of $\sigma$ is contained in a block of $\pi$. Hence $\Pi'$ is rank uniform with $R=R(\Pi')= (S(n+1,k+1))_{n,k=0}^\infty$. It is well known that $R$ is $\mathrm{TN}$. We will now find a combinatorial description of the polynomials $R_{n,k}(t)$, which gives another proof of the total nonnegativity of $R$. 

Recall that the \emph{chromatic polynomial} of a simple graph $G=(V,E)$ may be expressed as 
$$
\chi_G(t)= \sum_{k=1}^n S_G(k) \cdot (t)_k, \ \ \ |V|=n, 
$$
where $(t)_k = t(t-1)\cdots (t-k+1)$, and $S_G(k)$ is the number of partitions of $V$ {with $k$ blocks} such that no block contains an edge from $E$. The \emph{$\sigma$-polynomial} of $G$ is defined by 
$$
\sigma_G(t)= \sum_{k=1}^n S_G(k) \cdot t^k .
$$
Hence $\sigma_G(t)=\mathcal{S}(\chi_G(t))$,  where $ \mathcal{S} : \RR[t] \to \RR[t]$ is the invertible linear map defined by $ \mathcal{S}((t)_k) = t^k$ for all $k \in \NN$.

 Let $B_{n,k}$ be the graph on $[n]$ with edges $\{i, j\}$, $1\leq i<j \leq k$, and notice that $\chi_{B_{n+1,k+1}}(t)= t^{n-k}(t)_{k+1}$. 

\begin{figure}[H]
\centering
\begin{tikzpicture}[line cap=round,line join=round,>=triangle 45,x=1cm,y=1cm]
	\clip(-9,-2) rectangle (2.5,2);
	\draw [line width=1pt] (-8,1)-- (-6,-1);
	\draw [line width=1pt] (-6,-1)-- (-6,1);
	\draw [line width=1pt] (-6,1)-- (-8,1);
	\draw [line width=1pt] (-8,1)-- (-8,-1);
	\draw [line width=1pt] (-8,-1)-- (-6,1);
	\draw [line width=1pt] (-8,-1)-- (-6,-1);
	\draw [line width=1pt] (0,1)-- (-1,-1);
	\draw [line width=1pt] (-1,-1)-- (1,-1);
	\draw [line width=1pt] (1,-1)-- (0,1);
	\draw (-8.5,1.4) node[anchor=north west] {1};
	\draw (-6,1.4) node[anchor=north west] {2};
	\draw (-8.5,-1.1) node[anchor=north west] {3};
	\draw (-6,-1.1) node[anchor=north west] {4};
	\draw (-5.072039820114968,-0.2858820632730748) node[anchor=north west] {5};
	\draw (-4.052339710155575,-0.2749175459616835) node[anchor=north west] {6};
	\draw (-3.0545686348189647,-0.23105947671611815) node[anchor=north west] {7};
	\draw (-0.06125540880913355,1.5013342584837122) node[anchor=north west] {1};
	\draw (-1.3550684515533096,-1.1) node[anchor=north west] {2};
	\draw (1,-1.1) node[anchor=north west] {3};
	\draw (1.9891093284210437,-0.2749175459616835) node[anchor=north west] {4};
	\begin{scriptsize}
		\draw [fill=ududff] (-8,1) circle (2.5pt);
		\draw [fill=ududff] (-8,-1) circle (2.5pt);
		\draw [fill=ududff] (-6,1) circle (2.5pt);
		\draw [fill=ududff] (-6,-1) circle (2.5pt);
		\draw [fill=xdxdff] (-5,0) circle (2.5pt);
		\draw [fill=xdxdff] (-4,0) circle (2.5pt);
		\draw [fill=xdxdff] (-3,0) circle (2.5pt);
		\draw [fill=ududff] (-1,-1) circle (2.5pt);
		\draw [fill=xdxdff] (0,1) circle (2.5pt);
		\draw [fill=ududff] (1,-1) circle (2.5pt);
		\draw [fill=xdxdff] (2,0) circle (2.5pt);
	\end{scriptsize}
\end{tikzpicture}
\caption{The graphs of $B_{7,4}$ and $B_{4,3}$, respectively.}
\label{fig:partition-hyperplanes}
\end{figure}
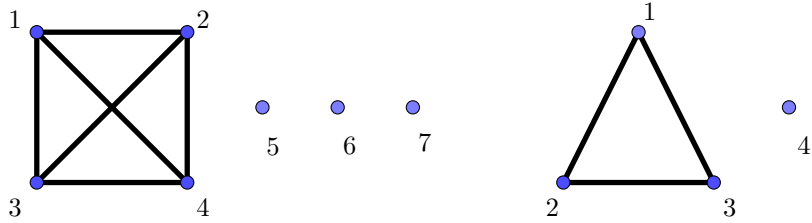

\begin{theorem}\label{partnk}
If $0\leq k \leq n$, then  
$$
R_{n,k}(t) = \sigma_{B_{n+1,k+1}}(t)/t=  \mathcal{S}\left(t^{n-k}(t)_{k+1}\right)/t.  
$$
Moreover 
$$
R_{n+1,k}(t)= R_{n+1,k+1}(t)+ (k+1) R_{n,k}(t), \ \ \ 0\leq k \leq n. 
$$
\end{theorem}

\begin{proof}
For $k=0$,
$$
\sigma_{B_{n+1,1}}(t)/t = \sum_{k=0}^n S(n+1,k+1)t^k = R_{n,0}(t). 
$$
It remains to prove the identity 
$$
\sigma_{B_{n+2,k+1}}(t)= \sigma_{B_{n+2,k+2}}(t)+ (k+1)\sigma_{B_{n+1,k+1}}(t),  
$$
or equivalently, 
$$
t^{n+1-k}(t)_{k+1}= t^{n-k}(t)_{k+2} + (k+1) t^{n-k}(t)_{k+1},
$$
which is trivially true.  
\end{proof}
We call $\Pi'$-posets \emph{partition posets}.  The $h$-vector of a partition poset $P$ of rank $n$ is the vector $(h_0(P), h_1(P), \ldots, h_n(P))$ for which 
$$
f_P(t)= \sum_{x \in P}t^{\rho(x)} = \sum_{k=0}^n h_k(P) \cdot R_{n,k}(t).
$$
Since dual partition lattices are lower semimodular, a partition poset is C-shellable if and only it satisfies (S1). 
Suppose (S1) is satisfied  for an ordering $y_1, \ldots, y_m$ of the facets of a partition poset $P$ of rank $n$. Identify $\langle y_k \rangle$ with $\Pi_{n+1}'$. 
An element $\pi_e \in \Pi_{n+1}'$ of rank $n-1$ has precisely $n-1$ blocks with one element each, and a single block $e=\{ i, j \}$. Hence we may identify $P_k=\langle y_k\rangle \cap (\langle y_1 \rangle \cup \cdots \cup \langle y_{k-1} \rangle)$ with a simple graph $G=G(P_k)=([n+1], E)$. Express the rank generating polynomial of 
$P_k$ as 
$$
f_{P_k}(t) = f_{P_{k-1}}(t)+ \sum_{x \in Q(G)} t^{\rho(x)},
$$
where 
$$
Q(G) =\langle y_k \rangle \setminus \left(P_{k-1} \cap \langle y_k \rangle\right) = \langle y_k \rangle \setminus  \cup_{e \in E} \langle \pi_{e} \rangle. 
$$
Now $\pi \in  Q(G)$ if and only if no block of $\pi$ contains an edge in $E$. Hence 
$$
\sum_{x \in Q(G)} t^{\rho(x)} = \sigma_G(t)/t = \mathcal{S}(\chi_G(t))/t.
$$
Next we examine properties of shellings that guarantee  that the $h$-vector of nonnegative coefficients. In light of the discussion  above and Theorem~\ref{partnk}, a reasonable condition is then that $\chi_G(t)$ has a nonnegative expansion in $\{ t^{n-k}(t)_{k+1} \}_{k=0}^n$. This is true for all graphs with at most $10$ vertices. This led us to ask if all graphs have a nonnegative expansion in $\{ t^{n-k}(t)_{k+1} \}_{k=0}^n$. However G. Royle and E. Gioan found that the complete bipartite $K_{7,7}$ fails to have a nonnegative expansion \cite{Royle}. Notice that this does not imply that there exist C-shellable partition posets that fail to have nonnegative $h$-vectors. 

\begin{problem}\label{partposprob}
    Do all C-shellable partition posets have nonnegative $h$-vectors? 
\end{problem}

Since wedon't have an answer to Problem \ref{partposprob}, we seek more restrictive versions of shellability for partition posets that imply nonnegative $h$-vectors. 

\emph{Chordal graphs} may be defined recursively as follows. The empty graph (the graph with no vertices or edges) is chordal. A nonempty graph $G=(V,E)$ is chordal if and only if there exists a vertex $v \in V$ such that the neighbors of $v$ form a clique and $G \setminus v$ is chordal.

\begin{definition}
\label{partshelling}
Let $P$ be a pure partition poset. A total order $y_1, \ldots, y_m$ on the facets of $P$ is called a \emph{$\Pi$-shelling} if  (S1) is satisfied, and 
\begin{itemize}
\item[(C)] for all $1<i \leq m$ the graph $G(y_i)$ defined above is chordal. 
\end{itemize}    
\end{definition}

\begin{lemma}
\label{chordal-nonnegative}
    Let $G$ be a chordal graph on $n$ vertices. Then the chromatic polynomial of $G$ has a nonnegative expansion in $\{ t^{n-k}(t)_{k} \}_{k=1}^n$. 
\end{lemma}

\begin{proof}
The proof is by induction on the number $n$ of vertices of $G$. The case when $n=1$ is clear. Suppose $n \geq 2$. By definition 
$G$ is obtained from a chordal graph $H$ on $n-1$ vertices by adding a new vertex $v$ whose neighbors $A$ form a clique. Hence 
$
\chi_G(t) = (t-a)\chi_H(t),
$
 where $|A|=a$. 
By induction, 
$$
\chi_H(t) = \sum_{k=1}^{n-1} \theta_k \cdot t^{n-1-k}(t)_{k}, \ \ \mbox{ where } \theta_k \geq 0 \mbox{ for all } k. 
$$
Since $\chi_H(k)=0$ for all $0 \leq k \leq a-1$, it follows that $\theta_k=0$ for all $0 \leq k \leq a-1$. Now 
$$
(t-a)t^{n-1-k}(t)_k= \left(1- \frac a k \right) t^{n-k}(t)_k + \frac a k t^{n-1-k}(t)_{k+1},
$$
which proves the lemma. 
\end{proof}

From Lemma~\ref{chordal-nonnegative} and the discussion above it we deduce:
\begin{corollary}\label{parpossh}
Let $P$ be a pure partition poset. If $P$ has a $\Pi$-shelling, then it has a nonnegative $h$-vector, and hence the chain polynomials of $P$ and its rank selected {subposet}s are real-rooted. 
\end{corollary}

Define a linear order on the set $\Pi_{n}^k$ of partitions of $[n]$ into $k$ blocks as follows. For Definition \ref{partshelling} to make sense, $\langle \Pi_{n}^k \rangle$ (the order ideal of $\Pi_n'$ generated by the partitions in $\Pi_{n}^k $) should be $\Pi$-shellable for all $k \leq n$. We shall now prove this.   Order the blocks $v_1, \ldots, v_k$ of $\pi \in \Pi_{n}^k$ so that $\max(v_1) < \cdots < \max(v_k)$. Order the elements within the blocks in decreasing order, and consider the permutation $\tau(\pi)= v_1v_2 \cdots v_k$. If $\pi, \pi' \in \Pi_{n}^k$ we write $\pi <_\ell \pi'$ if $\tau(\pi)$ comes before $\tau(\pi')$ in the lexicographical ordering. Hence $<_\ell$ is a total order. 

For $\pi \in \Pi_{n}^k$, define a graph $G(\pi)=(V,E)$, where $V=\{v_1,\ldots, v_k\}$ {is the ordered set of} blocks of $\pi$,  and $E$ is the set of all $\{v_i,v_j\}$ for which there exists a partition $\pi' \in \Pi_{n}^k$ for which $\pi' <_\ell \pi$ and   $\pi \wedge \pi'$ is equal to the partition $\pi[v_i,v_j]$ obtained from $\pi$ by merging the blocks $v_i$ and $v_j$. 

\begin{lemma}\label{vivj}
Let $\pi \in \Pi_{n}^k$. Then $\{v_i,v_j\}$, where $i<j$,  is not an edge in $G(\pi)$ if and only if 
\begin{itemize}
\item[(a)] $|v_1| = \cdots =|v_i|=1$, and 
\item[(b)] if $i < m \leq k$, then $x<y$ for all $x \in v_i$ and $y \in v_m$. 
\end{itemize}
\end{lemma}

\begin{proof}
We first prove that if (a) or (b) fails, then $\{v_i,v_j\}$ is an edge of $G(\pi)$. Suppose $|v_s|>1$ for some $s<i$ and consider the partition $\pi'$ obtained from $\pi$ by splitting $v_s$ into two blocks, and merging $v_i$ with $v_j$. Then $\pi' <_\ell \pi$ and $\pi \wedge \pi' = \pi[v_i,v_j]$. Hence $\{v_i,v_j\}$ is an edge. Similarly if $|v_i|>1$, then let $\pi'$ be the partition obtained from $\pi$ by moving the largest element of $v_i$ to $v_j$. Then $\pi' <_\ell \pi$ and $\pi \wedge \pi' = \pi[v_i,v_j]$. Hence $\{v_i,v_j\}$ is an edge.

Suppose (a) holds and (b) fails for some $m$. Let $\pi'$ be the partition obtained from $\pi$ by merging $v_i$ and $v_j$ and splitting $v_m$ into two blocks by making the smallest element of $v_m$ the only element in one of the blocks. Then $\pi' <_\ell \pi$ and $\pi \wedge \pi' = \pi[v_i,v_j]$. Hence $\{v_i,v_j\}$ is an edge.

Conversely if (a) and (b) hold, then $\{v_i,v_j\}$ is not an edge. 
\end{proof}

\begin{remark}\label{G-form}
Let $s$ be the maximal $i$ for which $v_i$ is not a vertex of an edge in $G(\pi)$. Then, by Lemma~\ref{vivj}, $v_j= \{j\}$ for all $1 \leq j \leq s$, and the graph $G(\pi)$ is isomorphic to $B_{k,k-s}$. 
\end{remark}

\begin{theorem}\label{pi-shell}
The order $<_\ell$ defines a $\Pi$-shelling of the facets of $\langle \Pi_{n}^k \rangle$. 
\end{theorem}

\begin{proof}
We need to prove that if $\pi' <_\ell \pi \in \Pi_{n}^k$, then $\pi' \wedge \pi \leq \pi[v_s,v_t]$  for some edge $\{v_s,v_t\}$ in $G(\pi)$. If $G(\pi)$ is the complete graph, then there is nothing to prove. Otherwise $\pi$ satisfies (a) and (b) in Lemma~\ref{vivj} for some maximal $i \geq 1$. Since $\pi' <_\ell \pi$, the blocks $\{1\}, \ldots, \{i\}$ of $\pi$ are also blocks of $\pi'$.  Hence  $\pi' \wedge \pi \leq \pi[v_s,v_t]$ for some $s>t >i$. Since $v_{i+1}, \ldots, v_k$ form a clique in $G(\pi)$, the proof follows. 
     
\end{proof}

\begin{corollary}
 Let $0\leq k <n$. The $h$-vector of  $\langle \Pi_{n+1}^{k+1} \rangle$ is equal to $$\big( (i+1) S(n-k+i,i+1)\big)_{i=0}^k.$$  
\end{corollary}

\begin{proof}
By Theorem~\ref{pi-shell},  the order $<_\ell$ defines a $\Pi$-shelling. By Remark \ref{G-form}, $h_{k-i}$ is equal to the number of partitions in $\Pi_{n+1}^{k+1}$ for which $\{j\}$ is a block of $\pi$ for all $1\leq j \leq i$, but $\{i+1\}$ is not a block. Hence 
$$
h_{k-i}= S(n-i+1, k-i+1)-S(n-i,k-i)= (k-i+1)S(n-i,k-i+1), 
$$
which proves the {statement}. 
\end{proof}

\noindent 
{\bf Acknowledgements.}  Thanks to an anonymous referee for several helpful suggestions that improved the paper greatly. PB is a Wallenberg Academy Scholar
  supported by the Knut and Alice Wallenberg Foundation, and the G\"oran Gustafsson foundation.

\bibliographystyle{abbrv}
\bibliography{referencess}

\end{document}